\documentclass[final]{siamart}

\usepackage{lipsum}
\usepackage{amsfonts}
\usepackage{amsmath} 
\usepackage{amssymb} 
\usepackage{mathrsfs}
\usepackage{graphicx}
\usepackage{epstopdf}
\usepackage{algorithmic}
\ifpdf
  \DeclareGraphicsExtensions{.eps,.pdf,.png,.jpg}
\else
  \DeclareGraphicsExtensions{.eps}
\fi

\newcommand{\TheTitle}{Gaussian approximations for transition paths in Brownian Dynamics} 
\newcommand{\TheAuthors}{Y. Lu, A. M. Stuart and H. Weber}

\headers{Gaussian approximations for transition paths}{\TheAuthors}

\title{{\TheTitle}\thanks{The authors are grateful Frank Pinski for helpful discussions and insights. YL is is supported by EPSRC as part of the MASDOC DTC at the University of Warwick with grant
No. EP/HO23364/1. The work of AMS is supported by DARPA, EPSRC
and ONR. The work of HW is supported by EPSRC and the Royal Society.}}

\author{
  Yulong Lu \thanks{Mathematics Institute, University of Warwick, Coventry, CV4 7DL, UK
    (\email{yulong.lu@warwick.ac.uk},  \email{hendrik.weber@warwick.ac.uk}).}
  \and
  Andrew Stuart \thanks{Computing \& Mathematical Sciences, California Institute of Technology, Pasadena, CA 91125, USA (\email{astuart@caltech.edu}).}
  \and
  Hendrik Weber\footnotemark[2]
}

\usepackage{amsopn}

\ifpdf
\hypersetup{
  pdftitle={Gaussian approximations for transition paths in Brownian Dynamics}, 
  pdfauthor={\TheAuthors}
}
\fi

\theoremstyle{nonumberplain}
\theoremheaderfont{\normalfont\itshape}
\theorembodyfont{\normalfont}
\theoremseparator{.}
\theoremsymbol{\proofbox}
\newtheorem{prooffoursix}{Proof of Proposition \ref{prop:compact-2}}
\newtheorem{proof45}{Proof of Proposition \ref{prop:sfeps}}
\newtheorem{proof4main}{Proof of Theorem \ref{thm:main}}
\newtheorem{proof4cor}{Proof of Corollary \ref{cor:minima}}
\renewcommand{\Delta}{\triangle}

\definecolor{darkblue}{rgb}{0,0,0.7}

\definecolor{darkgreen}{rgb}{0.01,0.75,0.24}

\def \Ee[#1]{\mathcal{E}^{\text{{#1}}}}
\def\R{\mathbf{R}}

\def \EE{\mathscr{E}} 

\def\pa[#1,#2]{\frac{\partial {#1}}{\partial {#2}} }

\def\idom[#1,#2,#3]{\int_{#1}\hspace{1pt} {#2} \hspace{1pt} \text{d}{#3}}
\def\res[#1,#2]{\left.{#1}\right|_{#2}}

\def\gt{\rightarrow}

\def\var[#1,#2]{\langle \delta \mathcal{E}^{\text{{#1}}}({#2}),v\rangle}
\def\vars[#1,#2,#3]{\langle \delta^2\mathcal{E}^{\text{{#1}}}({#2})v,{#3}\rangle}
\def\vard[#1,#2,#3,#4]{\langle \delta\mathcal{E}^{\text{{#1}}}({#2})-\delta\mathcal{E}^{\text{{#3}}}({#4}),v\rangle}

\def\P{\mathbb{P}}

\def\E{\mathbb{E}}
\def\Ec{\mathcal{E}}
\def\Rc{\mathcal{R}}
\def\Kc{\mathcal{K}}
\def\N{\mathbb{N}}

\newcommand{\bD}{\boldsymbol{\partial}_t^2}
\newcommand{\bSig}{\boldsymbol{\Sigma}}
\newcommand{\bLambda}{\boldsymbol{\Lambda}}
\newcommand{\bM}{\mathbf{M}}
\newcommand{\bF}{\mathbf{F}}
\newcommand{\bP}{\mathbf{P}}
\newcommand{\bR}{\mathbf{R}}
\newcommand{\bG}{\mathbf{G}}
\newcommand{\bA}{\mathbf{A}}
\newcommand{\bB}{\mathbf{B}}
\newcommand{\bC}{\mathbf{C}}
\newcommand{\bDD}{\mathbf{D}}

\newcommand{\cO}{\mathcal{O}}

\newcommand{\A}{\mathcal{A}}

\newcommand{\K}{\mathcal{K}}

\newcommand{\HH}{\mathcal{H}}

\newcommand{\eps}{\varepsilon}
\newcommand{\wgt}{\rightharpoonup}
\newcommand{\Tr}{\mathrm{Tr}}
\newcommand{\bH}{\mathbf{H}}

\newcommand{\bL}{\mathbf{L}}
\newcommand{\bI}{\mathbf{I}}
\newcommand{\bX}{\mathbf{X}}

\newcommand{\BV}{\textbf{BV}}

\newcommand{\Ms}{\mathscr{M}}

\newcommand{\be}{\begin{equation}}
\newcommand{\en}{\end{equation}}
\newcommand{\ben}{\begin{equation*}}
\newcommand{\enn}{\end{equation*}}
\newcommand{\bea}{\begin{aligned}}
\newcommand{\ena}{\end{aligned}}

\def\ba#1\ena{\begin{align}#1\end{align}}
\def\ban#1\enan{\begin{align*}#1\end{align*}}

\def\dpm {\prime\prime}

\theoremstyle{plain}
\newtheorem{thm}{Theorem}[section]
\newtheorem{defn}[thm]{Definition}
\newtheorem{lem}[thm]{Lemma}
\newtheorem{cor}[thm]{Corollary}
\newtheorem{prop}[thm]{Proposition}

\newtheorem{assumptions}[thm]{Assumptions}

\theoremstyle{remark}
\newtheorem{rem}[thm]{Remark}

\numberwithin{equation}{section}
\renewcommand{\theequation}{\thesection .\arabic{equation}}

\begin{document}

\maketitle

\begin{abstract}
This paper is concerned with transition paths within the framework
of the overdamped Langevin dynamics model of chemical reactions.  
We aim to give an efficient description of typical transition 
paths in the small temperature regime. We adopt a variational point of view 
and seek the best Gaussian approximation, with respect to Kullback-Leibler 
divergence, of the non-Gaussian distribution of the diffusion process. 
 We interpret the mean of this Gaussian approximation as the 
``most likely path'' and the covariance operator as a means to capture the typical fluctuations 
 around this most likely path. 

 We give an explicit expression for the Kullback-Leibler divergence in terms of the mean and the covariance operator for a natural class of Gaussian approximations and show the existence of 
 minimisers for the variational problem. 
Then the low temperature limit 
 is studied via $\Gamma$-convergence of the associated variational problem.
 The limiting functional consists of two parts: The first part only depends on the mean and
  coincides with the $\Gamma$-limit of the rescaled Freidlin-Wentzell rate functional. The second 
  part depends on both, the mean and the covariance operator and is minimized if the dynamics 
  are given by a time-inhomogenous Ornstein-Uhlenbeck process found by
linearization of the Langevin dynamics around the Freidlin-Wentzell minimizer.



\end{abstract}

\begin{keywords}
Transition path, 
Kullback-Leibler approximation, Onsager-Machlup functional, large deviations, Gamma-convergence.
\end{keywords}

\begin{AMS}
  28C20, 60G15, 60F10
\end{AMS}

\section{Introduction}
Determining the behavior of transition paths of complex 
mo- lecular dynamics is essential for understanding many problems in physics, chemistry and biology. Direct simulation of these systems can be prohibitively expensive, mainly due to the fact that the dynamical systems can exhibit the phenomenon of
{\em metastability}, which involves disparate time scales: the transition 
{\em between} metastable states is logarithmic in the inverse temperature, 
whilst 
fluctuations {\em within} the metastable states have durations which are 
exponential in the inverse
temperature. In many systems the interest is focused on the 
transition between metastable states and not the local fluctuations within
them. This paper addresses the problem of characterizing the most likely 
transition paths of molecular models of chemical reactions.

We focus on the Brownian dynamics model from molecular dynamics which
takes the form of a gradient flow in a potential, subject to small thermal
fluctuations:
\be\label{eq:LSDE}
d x(t) = -\nabla V(x(t))  dt+ \sqrt{2\eps} d W(t);
\en
we study the equation subject to the end-point conditions
\be\label{eq:LSDEBD}
x(0) = x_-, \quad x(T) = x_+.
\en 
Here $V:\R^d \gt \R$ is the potential function, $W$ is a standard Brownian motion in $\R^d$ and $\eps > 0$ is a small parameter related to the temperature of the thermal system. The Brownian dynamics model is widely used in the study
of molecular dynamics \cite{LSR10}. It is also referred to as the overdamped 
Langevin equation, and can be derived from the second order Langevin dynamics 
model, which has the form of damped-driven Newtonian dynamics with potential
energy $V$, by taking a large friction or a small mass limit; see 
\cite[Chapter 7, Exercise 8]{pavliotis2008multiscale} and \cite{LSR10}
for explicit derivations. 

Mathematically we understand the process $x(t), t\in [0, T]$ 
satisfying \eqref{eq:LSDE}, \eqref{eq:LSDEBD} to be the initial
value problem of \eqref{eq:LSDE} starting from $x(0) = x_-$, subject to
the  conditioning $x(T) = x_+$ \cite{HSV07}. We propose to study the
sample path of this conditioned process as a model for the temporal evolution of molecules making a transition between two atomistic configurations $x_\pm$. In this paper, we will assume that $x_{\pm}$ are critical points of $V$; indeed most interest focuses on the case where both endpoints are chosen to be local minima of $V$.

When the temperature $\eps$ is small and when the end-point condition on $x(T)$ is removed,  typical realisations of \eqref{eq:LSDE} exhibit fluctuations around the local minima of $V$ for long stretches of time (exponential in $\eps^{-1}$)
%
while the occasional rapid transitions between different minima occur on a much shorter time scale which is only logarithmic in $\eps^{-1}$. The difference between these time scales makes it difficult to sample transition paths when $\eps$ is small. As
an alternative to direct sampling, several notions of ``most likely transition paths'' have been proposed;
of particular interest here are the Freidlin-Wentzell 
and Onsager-Machlup theories.

In the zero temperature limit $\eps \gt 0$, the behaviour of transition paths can be predicted with overwhelming probability using Freidlin-Wentzell theory \cite{FSW12}. For any fixed $T$, the solution process $\{x(t), t\in[0, T]\}$ to  \eqref{eq:LSDE}, \eqref{eq:LSDEBD} satisfies a large deviation principle with rate (or action) functional given by
\be\label{eq:action}
\overline{S}_T(\varphi) := \frac{1}{4}\int_0^T |\varphi^\prime(t) + \nabla V(\varphi(t))|^2 dt
\en
with $\varphi \in H^1_{\pm}(0, T; \R^d) := \{x\in H^1(0, T; \R^d): x(0) = x_-, x(T) = x_+ \}$. Loosely speaking the large deviation principle states that for any small $\delta > 0$, the probability that the solution $x$ lies in a tube of width $\delta$ around a given path $\varphi$ is approximately given by
\be\label{eq:ldp}
\P\{x: \sup_{t\in [0, T]}|x(t) - \varphi(t)| \leq \delta\} \approx \exp(-\eps^{-1}\overline{S}_T(\varphi))
\en
for $\eps$ small enough. Here $\P$ denotes the law of the process defined in \eqref{eq:LSDE}, \eqref{eq:LSDEBD}.  The large deviation principle thus characterizes 
the exponential tail of the distribution of the transition paths; but what is of 
most interest to us is that it leads to a natural variational definition of the 
most likely path: the minimizer of the rate functional $\overline{S}_T$ can be 
interpreted as most likely path in the sense that the probability of a trajectory
in a small neighbourhood of this minimizer is exponentially larger in $\eps^{-1}$ 
than the probability of hitting neighbourhoods of any other paths.

In view of the boundary conditions \eqref{eq:LSDEBD}, one can rewrite the functional $\overline{S}_T$ as
\be\label{eq:rate}
\bea
\overline{S}_T(\varphi) & := \frac{1}{4}\int_0^T |\varphi^\prime(t) + \nabla V(\varphi(t))|^2 dt\\
& = \frac{1}{4}\int_0^T |\varphi^\prime(t)|^2 + |\nabla V(\varphi(t))|^2 dt + \frac{1}{2} \int_0^T \varphi^\prime(t) \cdot \nabla V(\varphi(t)) dt\\
& = \frac{1}{4}\int_0^T |\varphi^\prime(t)|^2 + |\nabla V(\varphi(t))|^2 dt + \frac{1}{2} \left((V(x_+) - V(x_-)\right).
\ena
\en
The last term in this expression only depends on the boundary conditions and not on the 
specific choice of $\varphi$. 
Hence minimizing $\overline{S}_T(\varphi)$ is equivalent to minimizing the following Freidlin-Wentzell functional
\be\label{eq:FW}
S_T(\varphi) := \frac{1}{4}\int_0^T |\varphi^\prime(t)|^2 + |\nabla V(\varphi(t))|^2 dt
\en
over $H^1_{\pm}(0, T; \R^d)$, and from now on we refer to the minimization of this functional as the Freidlin-Wentzell approach.
The Freidlin-Wentzell viewpoint has been enormously influential
in the study of chemical reactions. For example the string method 
\cite{ERV02,ERV05} is based on minimization of the action functional 
\eqref{eq:action} over paths parameterized by arc-length. See
the the review article \cite{Vanden-Eijnden14} for recent development of 
transition path theory. 


At finite temperature $\eps > 0$, optimal transition paths can be 
defined as minimizers of the Onsager-Machlup functional \cite{DB78}. This functional is defined by maximizing small ball probabilities for paths
$x(\cdot)$ solving \eqref{eq:LSDE}, \eqref{eq:LSDEBD}.
To be more precise, we denote by $\P_0$ the law of the Brownian bridge on $[0, T]$ connecting $x_-$ and $x_+$, corresponding to vanishing drift ($V = 0$) in \eqref{eq:LSDE},  \eqref{eq:LSDEBD}, which depends on $\eps.$ Then under certain conditions on $V$ (see (ii) of Remark \ref{rem:cond}), the measure $\P$ is absolutely continuous with respect to $\P_0$ and the Radon-Nikodym density is given by
\be\label{eq:mumu}
\frac{d\P}{d\P_0} (x) = \frac{1}{Z}\exp\left(-\frac{1}{2\eps}\int_0^T \Psi_\eps(x(t)) dt \right)
\en
where
\be\label{eq:psi}
\Psi_\eps(x) := \frac{1}{2} |\nabla V(x)|^2 - \eps \Delta V(x).
\en
Equation \eqref{eq:mumu} follows from Girsanov formula and It\^o's formula, see \cite[Section 2]{PST}. 
We define the Onsager-Machlup functional $I_\eps$ over the space $H^1_{\pm}(0,T;\R^d)$ by  
\be\label{eq:om}
I_\eps (x) := \frac{1}{2} \int_0^T \left( \frac{1}{2} |x^\prime(t)|^2 + \Psi_\eps(x(t)) \right) dt=S_T(x)-\frac{\eps}{2} \int_0^T \Delta V(x(t))dt. 
\en
In \cite{DB78} it was shown that for any $x_1, x_2 \in H^1_{\pm}(0,T;\R^d)$
$$
\lim_{\delta\gt 0} \frac{\P(B_\delta(x_1))}{\P(B_\delta(x_2))} = \exp \left(\frac{1}{\eps} (I_\eps(x_2) - I_\eps(x_1))\right)
$$
where $B_r(x)$ denotes a ball in $C([0,T];\R^d)$ with center $x$ and radius $r$. Hence for any fixed $x_2$, the above ratio of the small ball probability, as a function of $x_1$, is maximized at minimizers of $I_\eps$. In this
sense  minimizers of $I_\eps$ are analogous to
{\em Maximum A Posterior (MAP) estimators} which arise for the posterior
distribution $\P$ in Bayesian inverse problems; see \cite{DLSV13}.

The Onsager-Machlup functional \eqref{eq:om} differs from the Freidlin-Wentzell
functional only by the integral of the It\^o correction term $\eps \Delta V$. 
This difference 
arises because of the order in which the limits $\eps \to 0$
and $\delta \to 0$ are taken: in Freidlin-Wentzell theory the radius of
the ball $\delta$ is fixed and limit $\eps \to 0$ is studied while in
Onsager-Machlup theory $\eps$ is fixed and limit $\delta \to 0$ is studied.
For fixed $T > 0$, it is clear that $I_\eps(\varphi) \gt S_T(\varphi)$ as $\eps \gt 0$. Hence for fixed time scale $T$ the Onsager-Machlup theory agrees with the Freidlin-Wentzell theory in the low temperature limit. However, this picture can be different for large $T$, more precisely when $T \gt \infty$ as $\eps \gt 0$. 
 In fact,
as demonstrated in \cite{PS}, it is possible that when $T \gg 1$,
the MAP transition path spends a vast amount
of time at a saddle point of $V$ rather than at minima; moreover, for two paths with the same
energy barrier, the one passing through steeper confining walls is always preferred to the
other since a larger value of $\Delta V$ gives rise to a 
lower value of $I_\eps$. The discussion about the order of limits gives
a clue as to why this apparent contradiction occurs: by studying the limit
$\delta \to 0$ in Onsager-Machlup theory, for fixed temperature $\eps$, 
we remove entropic effects.


Both minimizing the Onsager-Machlup functional \eqref{eq:om} or finding MAP estimators 
are attempts to capture key properties of the distribution $\P$ by identifying a single most likely path. This   
can
be viewed as approximating the measure $\P$ by a Dirac measure in a well-chosen point. 
The key idea in this paper is to find better approximations to $\nu$ by
working in a larger class of measures than Diracs. We
will study the  best Gaussian approximations
with respect to Kullback-Leibler divergence. The mean of an optimal Gaussian should capture the concentration of the
target measure while its fluctuation characteristics are described by the 
covariance of the Gaussian. Furthermore the fluctuations can capture
entropic effects. Thus by using the Gaussian approximation we aim to
overcome the shortcomings of the Onsager-Machlup approach.
The idea of finding Gaussian approximations for non-Gaussian measures by means of the Kullback-Leibler divergence is not new. For
example, in the community of machine learning \cite{RW06}, Gaussian processes have been widely
used together with Bayesian inference for regression and prediction. Similar ideas have also been used to study models in ocean-atmosphere science \cite{MG11} and computational quantum mechanics \cite{bartok2010gaussian}. Recently, the problem of minimizing the Kullback-Leibler divergence between non-Gaussian measures and certain Gaussian classes was studied from the calculus of variation point of view \cite{PSSW15a} and numerical algorithms for Kullback-Leibler minimization were discussed in \cite{PSSW16}.

 The present paper builds on the theory developed in \cite{PSSW15a} and extends it to transition path theory. More specifically, the set of Gaussian measures for approximations is parameterized by a pair of functions $(m, \bA)$, where $m$ represents the mean and $\bA$ (defined in \eqref{eq:ou-1}) is used to define the covariance operator for the underlying Gaussian measure. For a fixed temperature $\eps$, the Kullback-Leibler divergence is  expressed as a functional $F_\eps$ depending on $(m, \bA)$ and existence of minimizers is shown in this framework. 
Then the asymptotic behaviour of the best Gaussian approximations in the low temperature limit is studied in terms of the $\Gamma$-convergence of the functionals $\{F_\eps\}$. The limiting functional (defined in \eqref{e:main-functional}) is identified as the sum of two parts. The first part, depending only on $m$, is identical to the $\Gamma$-limit of the rescaled Freidlin-Wentzell action functional, implying that for $\eps \to 0$ the most likely transition paths defined as the best Gaussian mean $m$ coincide with large deviation paths. The second part takes entropic effects into account and expresses the penalty for the fluctuations in terms of $\bA$; it vanishes if  $\bA = D^2 V (m(t))$ but this choice of 
$\bA$ is only admissible if  the Hessian $D^2 V (m)$ is positive definite. A strictly positive penalty occurs when $D^2 V (m(t))$ has a negative eigenvalue. Therefore minimizing the limiting functional amounts to selecting those optimal paths $m$ among the large deviation paths  that do not spend time in saddles or local maximizers. We stress that although at finite noise intensity $\eps > 0$ there 
is no explicit characterization of our most likely transition paths,  it is possible to approximately determine them numerically, as demonstrated in \cite{PSSW16}, see also Section \ref{sec:conclusion}.


This paper is organized as follows. In the next section we introduce
a time-rescaling of the governing Langevin equation, in terms of $\eps$, 
in which the undesirable effects of the Onsager-Machlup minimization are 
manifest; we also introduce some notation used throughout the paper. Furthermore, assumptions on the potential $V$ are discussed. In Section \ref{sec:kl}, we define the subset of Gaussian measures over which Kullback-Leibler 
minimization is conducted; the existence of minimizers to the variational 
problem is established at the end of this section. Then in Section \ref{sec:gammalim}, we study the low temperature limit of the Gaussian approximation using $\Gamma$-convergence. The main $\Gamma$-convergence result is given in Theorem \ref{thm:main}. Section \ref{sec:conclusion} discusses some important consequences of the $\Gamma$-convergence result, with emphasis on the link with theories of Freidlin-Wentzell and Onsager-Machlup. The proofs of Theorem \ref{thm:main} and some related results are presented in Section \ref{sec:proofs}.

\section{Set-up and Notation}

\subsection{Set-up}

As discussed in the previous section, the key issue which motivates our
work is the difference in behaviour between minimizers of the
Freidlin-Wentzell action and the Onsager-Machlup functional. This difference
is manifest when $T \gg 1$ and is most cleanly described by considering
the time scale $T = \eps^{-1}$. The $\Gamma$-limit of the Onsager-Machlup functional \eqref{eq:om} is studied, as $\eps \to 0$, under this time-rescaling,
in \cite{PST}; the limit exhibits the undesirable effects described in the preceding section.
Our objective is to characterize the $\Gamma$-limit for the variational 
problems arising from best Gaussian approximation with respect
to Kullback-Leibler divergences, under the same limiting process. 

Applying the time scaling $t \mapsto \eps^{-1} t$ to the equation \eqref{eq:LSDE} and noticing the boundary conditions \eqref{eq:LSDEBD}, yields 
 \be\label{eq:LSDE2}
 \begin{aligned}
& d x(t) = -\eps^{-1} \nabla V(x(t)) dt + \sqrt{2} d W(t), \\
& x(0) = x_-, \quad x(1) = x_+.
\end{aligned}
\en
The transformed SDE has an order one noise but a strong drift; it
will be our object of study throughout the remainder of the paper.
For technical reasons, we make the following assumptions on the potential $V$.

\begin{assumptions}\label{assump}
The potential $V$ appearing in \eqref{eq:LSDE2} satisfies:
\begin{itemize}

\item[ (A-1)] $V\in C^5(\R^d)$;

\item[(A-2)] the set of critical points 
\be\label{eq:criticalset}
\EE := \{x\in \R^d, \nabla V(x) = 0\}
\en is finite and the Hessian $D^2 V(x)$ is non-degenerate for any $x \in \EE$. 

\item[(A-3)] coercivity condition:
\be\label{eq:coer}
\exists R>0 \text{ such that } \inf_{|x| > R} |\nabla V(x)| > 0;
\en

\item[ (A-4)] growth condition: 
\be\label{eq:G}
\bea
\exists C_1, C_2 > 0 \text{ and } \alpha \in [0, 2) \text{ such that for all } x\in \R^d \text{ and } 1 \leq i,j,k \leq d,\\
 \limsup_{\eps\gt 0} \max\left(\big|\frac{\partial^3}{\partial x_{i} \partial x_{j} \partial x_{k}}\Psi_\eps(x)\big|, |\Psi_\eps(x)| \right)\leq C_1 e^{C_2|x|^\alpha};
\ena
\en

\item[(A-5)] $V(x) \gt \infty$ when $|x| \gt \infty$ and there exits $R > 0$ such that 
\be\label{eq:VP}
 2\Delta V(x) \leq  |\nabla V (x)|^2 \text{ for } |x| \geq R;
 \en

\item[(A-6)] monotonicity condition:
\be\label{eq:MON}
\exists R > 0 \text{ such that } |\nabla V(x_1)| \geq  |\nabla V(x_2)| \text{ if } |x_1| \geq |x_2| \geq R.
\en
\end{itemize}

\end{assumptions}
\begin{rem}\label{rem:cond}
\begin{enumerate}
\item[(i)] Conditions (A-2)-(A-3) are typical assumptions for
proving $\Gamma$-convergence results for Ginzburg-Landau and related
functionals \cite{FT89,L14}. The smoothness condition (A-1) is needed because our analysis involves a Taylor expansion of order three for $\Psi_\eps$. Furthermore, we will use conditions (A-4)-(A-6) to analyze the $\Gamma$-convergence problem in this paper. These assumptions will be employed to simplify the expectation term in the Kullback-Leibler divergence (see the expression \eqref{eq:kld2}).
\item[(ii)] The condition (A-5) is a Lyapunov type condition which guarantees that at small temperature ($\eps \leq 1$) the solution to the SDE in \eqref{eq:LSDE2} does not explode in finite time. The probability measure determined by this process is absolutely continuous with respect to the reference measure of the Brownian bridge. See \cite[Chapter 2]{R07} for more discussions about the absence of explosion. Moreover, by the definition of $\Psi_\eps$, (A-5) implies that for any $\delta \in \R$ there exists a constant $C> 0$ depending only $R$ and $\delta$ such that
\be\label{eq:psilb}
|\nabla V(x)|^2 - \eps \delta  \Delta V(x) \geq -C\eps \text{ for any } x\in \R^d.
\en
Such lower bound will be used to prove the compactness of the functionals of interest (see Proposition \ref{prop:compact-2}).
\item[(iii)] These conditions are not independent. For instance, the coercivity condition (A-3) can be deduced from the monotonicity condition (A-6) when $V(x)$ is non-constant for large $|x|$. Hence particularly (A-5) and (A-6) imply (A-3).
\item[(iv)] The set of functions satisfying conditions (A-1)-(A-7) is not empty: they are fulfilled by all polynomials. Therefore many classical potentials, such as the Ginzburg-Landau double-well potential $V(x) = \frac{1}{4} x^2 (1 -x)^2$ are included. $\qed$
\end{enumerate}
\end{rem}

For $\eps > 0$ we denote by $\mu_\eps$ the law of the above bridge process $x$ defined in \eqref{eq:LSDE2} and $\mu_0$ the law of the corresponding bridge for vanishing drift ($V = 0$) in \eqref{eq:LSDE2}. Then, by identical arguments to those yielding \eqref{eq:mumu},  $\mu_\eps$ is absolutely continuous with respect to $\mu_0$ and the Radon-Nikodym density is given by
\be\label{eq:mu}
\frac{d \mu_\eps}{d \mu_0} (x) = \frac{1}{Z_{\mu,\eps}} \exp \left(-\frac{1}{2\eps^2} \int_0^1 \Psi_\eps(x(t)) dt \right)
\en
where $\Psi_\eps$ is given by \eqref{eq:psi} and $Z_{\mu,\eps}$ is the normalization constant. Note that the extra factor $\frac{1}{\eps}$ with respect to \eqref{eq:mumu} is due to the time rescaling.

\subsection{Notation}\label{subsection:notation} Throughout the
paper, we use $C$ (or occasionally $C_1$ and $C_2$) to denote a generic positive constant which may change from one expression to the next and is independent of the temperature and any quantity of interest. We write $A \lesssim B$ if $A \leq CB$. Given an interval $I\subset \R$, let $L^p(I)$ and $W^{m, p}(I)$ with $ m\in \mathbf{N}, 1\leq p \leq \infty$ be the standard Lebesgue and Sobolev spaces of scalar functions respectively. Let $H^{m}(I) = W^{m, 2}(I)$. 
For $s\in [0,1]$, we set $H_0^s(I)$ to be the closure of $C_0^\infty(I)$ in $H^s(I)$ and equip it with the topology induced by $H^s(I)$. Define its dual space $H^{-s}(I) := (H^{s}_0(I))^\prime$. For $s > 1/2$, a function of $H^s_0(I)$ has zero boundary conditions.  Thanks to the Poincar\'e inequality, the $H^1$-semi-norm is an equivalent norm on $H^1_0(I)$. In the case that $I = (0,1)$, we simplify
the notations by setting $H^s_0 = H^s_0(0,1)$ and $H^{-s} = H^{-s}(0,1)$. 

We write scalar and vector variables in regular face whereas matrix-valued variables, function spaces for vectors and matrices are written in boldface.  Denote by $\mathcal{S}(d,\R)$ the set of all real symmetric $d\times d$ matrices and by $\bI_d$ the 
identity matrix of size $d$. 
Let $L^p(0,1; \R^d)$ and $L^p(0,1; \mathcal{S}(d,\R))$ be the spaces of vector-valued and symmetric matrix-valued functions with entries in $L^p(0,1)$ respectively. Similarly one can define $H^1(0,1; \R^d), H_0^s(0,1;\R^d)$ and $H^1(0,1; \mathcal{S}(d,\R))$. For simplicity, we use the same notation $\bL^p(0,1)$ (resp. $\bH^1(0,1)$) to denote $L^p(0,1; \mathcal{S}(d,\R))$ and $L^p(0,1; \R^d)$(resp. $H^1(0,1; \mathcal{S}(d,\R))$ and $H^1(0,1; \R^d)$).
For any $\bA = (A_{ij})\in L^p(0,1; \mathcal{S}(d,\R))$ with $1\leq p \leq \infty$, we define its norm
$$
\|\bA\|_{\bL^p(0,1)} := \left(\sum_{i=1}^d \sum_{j=1}^d \|A_{ij}\|^2_{L^p(0,1)}\right)^{\frac{1}{2}}.
$$
For $\bA = (A_{ij})\in H^1(0,1; \mathcal{S}(d,\R))$, the norm is defined by
$$
\|\bA\|_{\bH^1(0,1)} := \left(\sum_{i=1}^d \sum_{j=1}^d \|A_{ij}\|^2_{H^1(0,1)}\right)^{\frac{1}{2}}.
$$
 We also define $\bH^1_\pm(0, 1) := H^1_{\pm}(0,1; \R^d) := \{x\in H^1(0,1;R^d): x(0) = x_-, x(1) = x_+\}$. Denote by $\BV(I)$ the set of $\R^d$-valued functions of bounded variations on an interval $I \subset \R$.

 For matrices $\bA, \bB\in \mathcal{S}(d, \R)$ we write $\bA \geq \bB$ when $\bA - \bB$ is positive semi-definite. The trace of a matrix $\bA$ is denoted 
by $\Tr(\bA)$. Denote by $\bA^T$ the transpose of $\bA$ and by $|\bA|_F$ the Frobenius norm of $\bA$. Given $\bA\in \mathcal{S}(d, \R)$ with the diagonalized form $\bA = \bP^T \bLambda \bP$, we define the matrix matrix $|\bA| := \bP^T |\bLambda| \bP$. For matrices $\bA = (A_{ij})$ and $\bB = (B_{ij})$, we write
 $$\bA : \bB = \Tr(\bA \bB^T) = \sum_{i=1}^d\sum_{j=1}^d A_{ij} B_{ij}.$$
Define the matrix-valued operator $\bD := \partial_t^2 \cdot \bI_d$.
 For $a > 0$, we define
$$
\bL^1_a(0,1):= L^1_a(0,1;\mathcal{S}(d, \R)) = \left\{\bA\in L^1(0,1; \mathcal{S}(d,\R)): \bA(t) \geq a\cdot \bI_d\ a.e. \text{ on } (0,1)\right\}
$$
and 
$$
\bH^1_a(0,1) := H^1_a(0,1;\mathcal{S}(d, \R)) = \left\{\bA\in H^1(0,1; \mathcal{S}(d,\R)): \bA(t) \geq a\cdot \bI_d\ a.e. \text{ on } (0,1)\right\}
.$$
We write $\bA_n \wgt \bA$ in $\bL^1(0,1)$ when $\bA_n$ converges to $\bA$ weakly in $\bL^1(0,1)$. Let $\bH^1_0(0,1) = H^1_0(0,1;\R^d)$. Define
$\bH_0^s = \overbrace{H_0^s\times \cdots \times H_0^s}^{d}$
and let $\bH^{-s}$ be the dual. In addition, we define product
spaces $\HH := \bH^1_{\pm}(0,1) \times \bH^1(0,1), \HH_a := \bH^1_{\pm}(0,1) \times \bH^1_a(0,1), \mathcal{X} := \bL^1(0,1) \times \bL^1(0,1)$ and $\mathcal{X}_a := \bL^1(0,1) \times \bL^1_a(0,1)$.

  For a vector field $v = (v_1, v_2, \cdots, v_d)$, let $\nabla v = (\partial_i v_j)_{i,j=1,2,\cdots, d}$ be its gradient, which is a second order tensor (or matrix). Given a potential $V:\R^d \gt \R$, denote by $D^2 V$ the Hessian of $V$. Given a second order tensor $\mathbf{T} = (T_{ij})_{i,j=1,2,\cdots, d}$, we denote by $\nabla \mathbf{T}$ its gradient, which is a rank 3 tensor with $(\nabla \mathbf{T})_{ijk} = \frac{\partial \mathbf{T}_{ij}}{\partial x_k}$. In particular, we use $D^3 V$ to denote the gradient of the Hessian $D^2 V$. 

Finally we write
$\nu \ll \mu$ when the measure $\nu$ is absolutely continuous with respect to $\mu$ and write $\nu \perp \mu$
when they are singular. Throughout the paper, we denote by $N(m, \bSig)$ the Gaussian measure on $\bL^2(0,1)$ with mean $m$  
and covariance operator $\bSig$. Moreover, the Gaussian measures 
considered in the paper will always have the property that, almost surely,
draws from the measure are continous functions on $[0,1]$ and thus that
point-wise evaluation is well-defined. Given $h\in \bL^2(0,1)$, define the translation map $\mathcal{T}_h$ by setting $\mathcal{T}_h x = x + h$ for any $x\in \bL^2(0,1)$. Denote by $\mathcal{T}_h^\ast \mu$ the push-forward measure of a measure $\mu$ on $\bL^2(0,1)$ under the map $\mathcal{T}_h$.

\section{Kullback-Leibler Minimization}\label{sec:kl}

\subsection{Parametrization of Gaussian Measures}
In this subsection, we describe  the parametrization of the 
Gaussian measures that we use in our Kullback-Leibler minimization. 
To motivate our choice of parameterization we consider the SDE
\eqref{eq:LSDE2}.
This equation has order-one noise, but with a strong gradient-form 
drift which will, most of the time, constrain the sample path to 
the neighbourhood of critical points of $V$. The size of the
neighbourhood will be defined by small fluctuations whose size 
scales with $\eps^{\frac12}$. To capture this behaviour 
we seek an approximation 
to \eqref{eq:LSDE2} of the form $x = m + z$, where $m$ is a path connecting
$x_{\pm}$ in unit time and where $z$ describes the small fluctuations.
We aim to find $m$ from an appropriate class of functions, and $z$ as
 time-inhomogenous Ornstein-Uhlenbeck process
\be\label{eq:ou-1}
\begin{aligned}
d z(t) & = -\eps^{-1}\bA(t) z(t)dt + \sqrt{2}d W(t),\\
 z(0) &= z(1) = 0.
\end{aligned}
\en
The time-dependent functions $(m,\bA)$ become our unknowns. For subsequent discussions, we require
$m\in  \bH^1_\pm(0, 1).$ 
For $\bA$ we assume that 
$\bA \in \bH^1(0,1)$, i.e. $\bA\in H^1(0,1;\R^{d\times d})$ and  $\bA(t)$ is symmetric for any $t\in (0,1)$. The symmetry property will simplify the calculation of the 
change of measures 
below, and will also be helpful in estimating the Green€™s functions used to show the $\Gamma$-convergence in Section \ref{sec:gammalim}.

Let $\overline{\nu}_\eps$ be the distribution of the process $z$ defined by \eqref{eq:ou-1} and let $\overline{\mu}_0$ be the corresponding Brownian bridge (with $\bA = 0$).
The lemma below shows that $\overline{\nu}_\eps$ is a centred Gaussian with the covariance operator given by the inverse Schr\"odinger 
operator $\bSig_\eps := 2(-\bD + \bB_\eps)^{-1}$ with $\bB_\eps = \eps^{-2}\bA^2 - \eps^{-1}\bA^\prime$. Here $2(-\bD + \bB_\eps)^{-1}$ denotes the inverse of the Schr\"odinger oprator $\frac{1}{2} (-\bD + \bB_\eps)$ with Dirichlet boundary condition. 
Let $\bM_\eps(t; s)$ be the fundamental matrix satisfying
\be\label{eq:M}
\frac{d}{d t} \bM_\eps(t, s) = -\eps^{-1} \bA(t) \bM_\eps(t, s), \quad \bM_
\eps(s,s) = \bI_d.
\en
\begin{lem}\label{lem:RNder0}
Let $A\in \bH^1(0,1)$. Then the Radon-Nikodym density of $\overline{\nu}_\eps$ with respect to  $\overline{\mu}_0$ is given by
\be\label{eq:RNder0}
\frac{d\overline{\nu}_\eps}{d\overline{\mu}_0} (z)  = \frac{1}{Z_{\overline{\nu},\eps}} \exp\left(-\frac{1}{4}\int_0^1 z(t)^T \bB_\eps(t) z(t) dt \right) 
\en
where $\bB_\eps = \eps^{-2} \bA^2 - \eps^{-1} \bA^\prime$ and the normalization constant
\be\label{eq:norm1}
Z_{\overline{\nu}, \eps} =\exp\left(-\frac{1}{2\eps} \int_0^1 \Tr(\bA(t)) dt\right)\cdot \left(\int_0^1 \overline{\bM}_\eps(t)\overline{\bM}^T_\eps(t)dt\right)^{-1/2},
\en
where $\overline{\bM}_\eps(t) = \bM_\eps(1, t)$.
 It follows that $\overline{\nu}_\eps = N(0, 2(-\bD + \bB_\eps)^{-1})$.
\end{lem}
\begin{proof}
Let $z$ be the unconditioned Ornstein-Uhlenbeck process that satisfies
\be\label{eq:norconstsde}
d z(t) = -\eps^{-1} \bA(t) z(t) dt  + \sqrt{2}d W(t), \quad z(0) = 0.
\en
 Denote by $\widetilde{\nu}_\eps$ the law of $z(t),t\in [0,1]$ solving
\eqref{eq:norconstsde} 
and by $\widetilde{\mu}_0$ the law of the process $\sqrt{2}W(t)$. 
 It follows from Girsanov's theorem that
 \be
\frac{d \widetilde{\nu}_\eps}{d \widetilde{\mu}_0}(z) = \exp\left(-\frac{1}{2\eps}\int_0^1 \bA(t)z(t)\cdot dz(t) - \frac{1}{4\eps^2}\int_0^1 |\bA(t) z(t)|^2 dt \right).
\en
Simplifying the exponent on the right side of the
above by It\^o's formula gives
\be\label{eq:density1}
\frac{d \widetilde{\nu}_\eps}{d \widetilde{\mu}_0}(z) = \exp\left(-\frac{1}{4}\int_0^1 z(t)^T \bB_\eps(t) z(t) dt + \frac{1}{2\eps}\int_0^1 \Tr(\bA(t))dt - \frac{1}{4\eps}z(1)^T \bA(1) z(1)\right). 
\en
After conditioning on $z(1) = 0$ and using \cite[Lemma 5.3]{HSV07}, \eqref{eq:RNder0} follows from \eqref{eq:density1}. We now calculate the normalization constant $Z_{\overline{\nu},\eps}$. Let $\rho_1$ be the density of the distribution of $z(1)$ under the measure $\widetilde{\nu}_\eps$. Let $\widetilde{\mu}_y$ be law of the conditioned process $ (\sqrt{2}W(t) | \sqrt{2}W(1) = y)$. From  \eqref{eq:density1}, one can see that for any bounded measurable function $f:\R^d \gt \R$, 
\be\bea\label{eq:densityrho}
& \E^{\rho_1}[f(z(1))]= \E^{\widetilde{\nu}_\eps}[f(z(1))]\\
 & = \E^{\widetilde{\mu}_0} \bigg[ \exp\left(-\frac{1}{4}\int_0^1 z(t)^T \bB_\eps(t) z(t) dt \right)  \\
 & \qquad \times \exp\left(\frac{1}{2\eps}\int_0^1 \Tr(\bA(t))dt - \frac{1}{4\eps}z(1)^T \bA(1) z(1)\right) f(z(1)) \bigg]\\
 & = \frac{1}{(4\pi)^{d/2}}\int_{\R^d} \exp\left(\frac{1}{2\eps}\int_0^1 \Tr(\bA(t))dt - \frac{1}{4\eps}y^T \bA(1) y - \frac{1}{4}|y|^2\right) f(y) 
\\ & \qquad \times \E^{\widetilde{\mu}_y} \left[\exp\left(-\frac{1}{4}\int_0^1 z(t)^T \bB_\eps(t) z(t) dt \right)\right] dy,
\ena
\en
where we have used the fact that $z(1) \sim N(0, 2\cdot\bI_d)$ when $z$ is distributed according to $\widetilde{\mu}_0$.
 Then we can read from \eqref{eq:densityrho} that 
\be\label{eq:lambda1}
\rho_1(0) =  \E^{\overline{\mu}_0} \Big[ \exp\Big( - \frac{1}{4}\int_0^1 z(t)^T \bB_\eps(t) z(t) \, dt \Big) \Big]  \exp\Big( \frac{1}{2\eps}  \int_0^1 \Tr(\bA(t)) \, dt \Big) \frac{1}{(4\pi)^{d/2}}.
\en

On the other hand, we know from Appendix \ref{appendix-B} that the solution $z(t)$ of \eqref{eq:norconstsde} can be represented as
$$
z(t) = \sqrt{2} \int_0^t \bM_\eps(t, s) dW(s),
$$
where $\bM_\eps$ is the fundamental matrix (see Definition \ref{def:ch5fundm}). In particular, by It\^o's isometry the random variable $z(1)$ is a centred Gaussian with covariance 
$$
\E[z(1) z(1)^T] = 2\int_0^1 \overline{\bM}_\eps(t) \overline{\bM}_\eps(t)^T dt,
$$
where $\overline{\bM}_\eps(t) = \bM_\eps(1, t)$.
Therefore we obtain an alternative expression for $\rho_1$, namely
 \be\label{eq:lambda2}
\rho_1(0) =  \frac{1}{(4 \pi)^{d/2}} \left[\det \Big( \int_0^1 \overline{\bM}_\eps(t) \overline{\bM}_\eps(t)^T \,dt \Big)\right]^{-\frac12}.
\en
Comparing the expressions \eqref{eq:lambda1} and \eqref{eq:lambda2} yields \eqref{eq:norm1}.
%
%
Finally, by the same arguments used in the proof of \cite[Lemma C.1]{PSSW15a}, one can see that $\overline{\nu}_\eps = N(0, 2(-\bD + \bB_\eps)^{-1})$.
\end{proof}

We remark that the covariance operator $\bSig_\eps = 2(-\bD + \bB_\eps)^{-1}$ is bounded from $\bL^2(0,1)$ to $\bH^{2}(0,1)$ and is trace-class on $\bL^2(0,1)$; see Lemma~\ref{lem:schodinger} and Remark~\ref{rem:schodinger}. 
The sample paths $z$ are almost surely continuous and the covariances are given by
\be\label{eq:greenexp}
\E^{\overline{\nu}_\eps}[z(t) z(s)^T] = 2\bG_\eps(t,s), \quad t,s \in [0,1].
\en
Here $\bG_\eps(t, s)$ is the Green's tensor (fundamental matrix) of the elliptic operator $(-\bD + \bB_\eps)$ under Dirichlet boundary conditions, i.e. for any $s \in (0,1)$,
 \be\label{eq:grt}
 \bea
& \left(-\bD + \eps^{-2}\bA^2(\cdot) - \eps^{-1} \bA^\prime(\cdot)\right) \bG_\eps(\cdot, s) = \delta(\cdot - s)\cdot\bI_d,\\
& \bG_\eps(0, s) = \bG_\eps(1, s) =  0.
\ena
 \en

With a description of the centered fluctuation process $z$ in hand we now move on to discuss the non-centered process
 $x = m + z$, whose law is denoted by $\nu_\eps$. It is clear that $\nu_\eps = N(m, \bSig_\eps)$.
Because of \eqref{eq:ou-1}, $\nu_\eps$ can also be viewed as the law of the following conditioned Ornstein-Uhlenbeck process
\be\label{ou-2}
\begin{aligned}
& dx(t) = \left(m^\prime(t) - \eps^{-1}\bA(t)\big(x(t) - m (t)\big) \right)d t + \sqrt{2} d W(t), \\
& x(0) = x_-, \quad x(1) = x_+.
\end{aligned}
\en
 Hence the Gaussian measure $\nu_\eps$ is
parametrized by the pair of functions $(m, \bA)$. To conclude, recalling the space $\HH_a = \bH^1_{\pm}(0,1) \times \bH^1_a(0,1)$, we define the family of Gaussian measures as
\be\label{eq:GM1}
\A =  \Big\{N\left(m, 2(-\bD + \bB_\eps)^{-1}\right) : (m, \bA)\in \HH\Big\}
\en
where $\bB_\eps = \eps^{-2} \bA^2 - \eps^{-1} \bA^\prime$. For $a > 0$, we denote by $\A_a$ the set of Gaussian measures defined in the same way as \eqref{eq:GM1} but with $\HH$ replaced by $\HH_a$.

\subsection{Calculations of Kullback-Leibler divergence}
To quantify the closeness of probability
measures, we use the Kullback-Leibler divergence, or relative entropy.
Given two probability measures $\nu$ and $\mu$, with $\nu$ absolutely
continuous with respect to $\mu$, the Kullback-Leibler divergence of $\nu$ and $\mu$ is
$$D_{\text{KL}}(\nu || \mu) = \E^{\nu}\log\left(\frac{d\nu}{d\mu}\right)$$
where $\E^\nu$ denotes the expectation taken with respect to the measure $\nu$;
if $\nu$ is not absolutely continuous with respect to $\mu$, 
then the Kullback-Leibler divergence is defined as $+\infty$. Sometimes it is
convenient to evaluate the Kullback-Leibler divergence through a reference measure $\mu_0$. If the measures $\mu, \nu$ and $\mu_0$ are mutually equivalent, then the Kullback-Leibler divergence can be expressed as
\be
\label{eq:Klg}
D_{\text{KL}}(\nu || \mu) 
= \E^\nu \log\left(\frac{d\nu}{d\mu_0}\right) - \E^\nu \log\left(\frac{d\mu}{d\mu_0}\right).
\en
In this section, we calculate the Kullback-Leibler divergence between the non-Gaussian measure $\mu_\eps$ (defined by \eqref{eq:mu})  and the parametrized Gaussian measure $\nu_\eps = N(m, \bSig_\eps)$. Recall that $\overline{\nu}_\eps$ is the law of the time-inhomogeneous Ornstein-Uhlenbeck process \eqref{eq:ou-1}.  Recall also that $\mu_0$ is the law of the Brownian bridge process corresponding to vanishing drift
in the SDE \eqref{eq:LSDE2}. It is clear that $\mu_0 = N(m_0, 2(-\bD)^{-1})$ with $m_0(t) = x_- (1 - t) + x_+ t$. In order to evaluate the above Kullback-Leibler divergence by using \eqref{eq:Klg}, we need to calculate the Radon-Nikodym derivative $d\nu_\eps / d\mu_0$.  


\begin{lem}\label{lem:RNder1}
Let $m\in \bH^1_{\pm}(0,1)$ and $\bA\in \bH^1(0,1)$. Then the Radon-Nikodym density of $\nu_\eps$ with respect to  $\mu_0$ is given by
\be\label{RNder1}
\frac{d\nu_\eps}{d\mu_0} (x)  = \frac{1}{Z_{\nu,\eps}} \exp\left(-\Phi_{\nu,\eps}(x)\right)
\en
where 
\be
\bea
\Phi_{\nu,\eps}(x) & = \frac{1}{4}\int_0^1 (x(t) - m(t))^T \bB_\eps(t) (x(t) - m(t)) d t \\
& - \frac{1}{2}\int_0^1 m^{\prime}(t)\cdot dx(t) + \frac{1}{4}\int_0^1 |m^\prime (t)|^2 dt.
\ena
\en
and the normalization constant
\be\label{eq:norm2}
Z_{\nu, \eps} = \exp\left(\frac{|x_1 - x_0|^2}{4}\right) \cdot \exp\left(-\frac{1}{2\eps} \int_0^1 \Tr(\bA(t)) dt\right)\cdot \left(\int_0^1 \overline{\bM}_\eps(t)\overline{\bM}^T_\eps(t)dt\right)^{-1/2},
\en
where $\overline{\bM}_\eps(t) = \bM_\eps(1, t)$.
\end{lem}
\begin{proof}
First by definitions of $\overline{\nu}_\eps$ and $\overline{\mu}_0$, we know that $\nu_\eps = \mathcal{T}_{m}^\ast \overline{\nu}_\eps$ and $\mu_0 = \mathcal{T}_{m_0}^\ast \overline{\mu}_0$. Then we have
\be\label{eq:density2}
\frac{d \nu_\eps}{d \mu_0}(x)  = \frac{d \mathcal{T}_{m}^\ast \overline{\nu}_\eps}{d \mathcal{T}_{m_0}^\ast \overline{\mu}_0}(x) = \frac{d \mathcal{T}_{m}^\ast \overline{\nu}_\eps}{d \mathcal{T}_{m}^\ast \overline{\mu}_0}(x) \cdot \frac{d \mathcal{T}_{m}^\ast \overline{\mu}_0} {d \mathcal{T}_{m_0}^\ast \overline{\mu}_0}(x).
\en
Observe that for any Borel set $A\subset \bL^2(0,1)$, 
$$
\mathcal{T}_{m}^\ast \overline{\nu}_\eps(A) =\overline{\nu}_\eps(A - m) = \E^{\overline{\mu}_0} \left[\frac{d \overline{\nu}_\eps}{d \overline{\mu}_0}(x) \mathbf{1}_{A - m}(x)\right] =
\E^{\mathcal{T}_m^\# \overline{\mu}_0} \left[\frac{d \overline{\nu}_\eps}{d \overline{\mu}_0}(x - m) \mathbf{1}_{A}(x)\right]. 
$$
This together with Lemma \ref{lem:RNder0} implies that
\be\bea\label{eq:density3}
\frac{d \mathcal{T}_{m}^\ast \overline{\nu}_\eps}{d \mathcal{T}_{m}^\ast \overline{\mu}_0}(x) & = \frac{d \overline{\nu}_\eps}{d\overline{\mu}_0}(x - m).\\
& =
 \frac{1}{Z_{\overline{\nu},\eps}} \exp\left(- \frac{1}{4}\int_0^1 (x(t) - m(t))^T \bB_\eps(t) (x(t) - m(t)) dt \right).
\ena
\en
Since $m\in \bH^1_\pm(0,1)$, $m - m_0 \in \bH_0^1(0,1)$ and hence $\mathcal{T}_{m}^\ast \overline{\mu}_0 \ll \mathcal{T}_{m_0}^\ast \overline{\mu}_0$. Furthermore, by the Cameron-Martin formula we have
\be\label{eq:density4}
\bea
\frac{d \mathcal{T}_{m}^\ast \overline{\mu}_0} {d \mathcal{T}_{m_0}^\ast \overline{\mu}_0}(x) & = \exp\Big(\frac{1}{2}\int_0^1\big(m^\prime(t) - m_0^\prime(t)\big)\cdot d \big(x(t) - m(t)\big) \\
& \quad\quad - \frac{1}{4}\int_0^1 |m^\prime(t) - m^\prime_0(t)|^2 dt \Big).
\ena
\en
Recall that $m_0(t) = x_-(1 - t) + x_+ t$. Using the fact that $x(0) = x_-, x(1) = x_+$ when $x$ is distributed according to $\mathcal{T}_m^\ast \overline{\mu}_0$ (or $\mathcal{T}_{m_0}^\ast \overline{\mu}_0$), we can simplify the exponent of above as follows: 
\be\bea\label{eq:density5}
& \frac{1}{2}\int_0^1\big(m^\prime(t) - m_0^\prime(t)\big)\cdot d \big(x(t) - m_0(t)\big) - \frac{1}{4}\int_0^1 |m^\prime(t) - m^\prime_0(t)|^2 dt \\
& = \frac{1}{2} \int_0^1 m^\prime(t) \cdot d x(t) -\frac{1}{4}\int_0^1 |m^\prime(t)|^2 dt - \frac{1}{2}\int_0^1 m_0^\prime(t) \cdot d \big(x(t) - m_0(t)\big)\\
& \quad\quad -\frac{1}{4}\int_0^1 |m_0^\prime(t)|^2 dt\\
& = \frac{1}{2} \int_0^1 m^\prime(t) \cdot d x(t) -\frac{1}{4}\int_0^1 |m^\prime(t)|^2 dt - \frac{|x_+ - x_-|^2}{4}.
\ena
\en
Hence one can obtain \eqref{RNder1} from \eqref{eq:density2}-\eqref{eq:density5} where the normalization constant
$$
Z_{\nu, \eps} = Z_{\overline{\nu},\eps} \cdot \exp\left(\frac{|x_+ - x_-|^2}{4}\right).
$$
This together with \eqref{eq:norm1} implies \eqref{eq:norm2}.
\end{proof}

According to the definition of $\mu_\eps$ (given by \eqref{eq:mu}), Lemma \ref{lem:RNder1} and the expression \eqref{eq:Klg} for the Kullback-Leibler divergence we obtain that  
\be
\begin{aligned}\label{eq:kld}
 D_{\text{KL}} (\nu_\eps || \mu_\eps) = \widetilde{D}_{\text{KL}}(\nu_\eps || \mu_\eps)  -\frac{|x_1 - x_0|^2}{4} + \log(Z_{\mu,\eps}),
\end{aligned}
\en
where  
\be
\bea\label{eq:kld2}
\widetilde{D}_{\text{KL}}(\nu_\eps || \mu_\eps) & = 
\frac{1}{2\eps^2}\E^{\overline{\nu}_\eps} \int_0^1 \Psi_\eps(z(t) + m(t))dt + \frac{1}{4} \int_0^1 |m^\prime (t)|^2 dt  \\
& - \frac{1}{4} \E^{\overline{\nu}_\eps}\left[\int_0^1 z(t)^T \bB_\eps(t) z(t) d t\right] + \frac{1}{2\eps}\int_0^1 \Tr(\bA(t)) dt\\
&  + \frac{1}{2}\log \left(\det\left(\int_0^1 \overline{\bM}_\eps(t)\overline{\bM}_\eps(t)^T dt\right)\right).
\ena
\en 
Here $\overline{\nu}_\eps = N(0, 2(-\bD + \bB_\eps)^{-1})$ and 
$\overline{\bM}_\eps(t) =  \bM_\eps(1,t)$ with $\bM_\eps$ defined by \eqref{eq:M}.
The form of $\widetilde{D}_{\text{KL}}(\nu_\eps || \mu_\eps)$ 
is interesting: the first two terms comprise a ``fattened'' version 
of the Onsager-Machlup functions \eqref{eq:om}, 
where the fattening is characterized by the
entropic fluctuations of the process $z$. 
The remaining terms penalize those entropic contributions. 
This characterization will be particularly clear in the
small noise limit -- see the discussion in Section~\ref{sec:conclusion}.

\subsection{Variational Problem}

Recall the set of Gaussian measures
$$
\A = \Big\{N(m, 2(-\bD + \bB_\eps)^{-1}): (m, \bA) \in \HH\Big\}
$$
where $\bB_\eps  = \eps^{-2}\bA^2 - \eps^{-1} \bA^\prime$ and that the set $\A_a$ is defined in the same way with $\HH$ replaced by $\HH_a$ for some $a > 0$. Given the measure $\mu_\eps$ defined by \eqref{eq:mu}, i.e. the law of transition paths,
we aim to find optimal Gaussian measures $\nu_\eps$ from $\A$ or $\A_a$ minimizing the Kullback-Leibler divergence $D_{\text{KL}}(\nu_\eps|| \mu_\eps)$. To that end, first in view of \eqref{eq:kld}, the constants $\frac{|x_1 - x_0|^2}{4}$ and $\log(Z_{\mu,\eps})$ can be neglected in the minimization process since they do not depend on the choice of $\nu_\eps$. Hence we are only concerned with minimizing the modified Kullback-Leibler divergence $\widetilde{D}_{\text{KL}}(\nu_\eps || \mu_\eps)$.  
   Furthermore, instead of minimizing $\widetilde{D}_{\text{KL}}(\nu_\eps || \mu_\eps)$,    we consider the variational problem
 \be\label{eq:min2}
 \inf_{\nu\in \A} \left( \eps \widetilde{D}_{\text{KL}}(\nu_\eps || \mu_\eps) + \eps^{\gamma}\|\bA\|^2_{\bH^1(0,1)}\right), 
 \en
 where $\gamma > 0$ and $\A$ is given by \eqref{eq:GM1}. We will also study the minimization problem over the set $\A_a$. The reasons why the problem \eqref{eq:min2} is of interest to us are the following. First, multiplying $\widetilde{D}_{\text{KL}}(\nu_\eps || \mu_\eps)$ by $\eps$ does not change the minimizers. Yet after this scaling the $m$-dependent terms of $\widetilde{D}_{\text{KL}}(\nu_\eps || \mu_\eps)$ (the first two terms on the right hand side of \eqref{eq:kld2}) and the $\bA$-dependent terms (middle line of \eqref{eq:kld2}) are well-balanced since they are all order one quantities with respect to $\eps$. 
Moreover, the regularization term $\eps^{\gamma}\|\bA\|^2_{\bH^1(0,1)}$ is necessary because the matrix $\bB_\eps$, along any infimizing sequence for
$\eps\widetilde{D}_{\text{KL}}(\nu_\eps || \mu_\eps)$, will only converge 
weakly and the minimizer may not be attained in $\A$. This issue is
illustrated in \cite[Example 3.8 and Example 3.9]{PSSW15a} and a 
similar regularization is used there. 

\begin{rem}
The normalization constant $Z_{\mu, \eps}$ in \eqref{eq:kld} is dropped in our minimization problem. This is one of the advantages of quantifying measure approximations by means of the Kullback-Leibler divergence. However, understanding the asymptotic behavior of $Z_{\mu,\eps}$ in the limit $\eps\gt 0$ is quite important, even though this is difficult. In particular, it allows us to study the asymptotic behavior of the scaled Kullback-Leibler divergence $\eps D_{\text{KL}}(\nu_\eps || \mu_\eps)$, whereby quantitative information on the quality
of the Gaussian approximation in the small temperature limit can be extracted. In the next section we study behavior of the minimizers of $F_\eps$ in the limit $\eps\gt 0$; 
we postpone study of $\eps D_{\text{KL}}(\nu_\eps || \mu_\eps)$, which
requires analysis of $Z_{\mu,\eps}$ in the limit $\eps\gt 0$, 
to future work.
$\qed$
\end{rem}

\begin{rem}
We  choose the small weight $\eps^{\gamma}$ with some $\gamma > 0$ in front of the regularization term with the aim of weakening the contribution from the regularization so that it disappears in the limit $\eps \gt 0$. For 
the study of the $\Gamma$-limit of $F_\eps$, we will consider $\gamma\in (0,\frac{1}{2})$; see Theorem \ref{thm:main} in the next section.$\qed$
\end{rem}
\begin{rem}
The Kullback-Leibler divergence is not symmetric in its
arguments. We do not study $\widetilde{D}_{\text{KL}}(\mu_\eps || \nu_\eps)$ because 
minimization of this functional over the class of Gaussian measures leads simply to moment matching and this is not approprpriate for problems with
multiple minimizers, see \cite[Section 10.7]{B06}.
$\qed$
\end{rem}
The following theorem establishes the existence of minimizers for the problem \eqref{eq:min2}. 
\begin{thm}\label{thm:opt}
Given the measure $\mu_\eps$ defined by \eqref{eq:mu} with fixed $\eps > 0$. There exists at least one measure $\nu \in \A$ (or $\A_a$) minimizing the functional  
\be\label{eq:minifunc}
\nu \mapsto \eps \widetilde{D}_{\text{KL}}(\nu || \mu_\eps) + \eps^{\gamma}\|\bA\|^2_{\bH^1(0,1)}. 
\en
over $\A$ (or $\A_a$).
\end{thm}
\begin{proof}
We only prove the theorem for the case where the minimizing problem is defined over $\A_a$ since the other case can be treated in the same manner.  
First we show that the infimum of \eqref{eq:minifunc} over $\A_a$ is finite for any fixed $\eps > 0$. In fact, consider $\bA^\ast =  a\cdot \bI_d$ with $a > 0$ and $m^\ast$ being any fixed function in $\bH^1_\pm(0,1)$. Then we show that $F(m^\ast, \bA^\ast)$ is finite. For this, by the formula \eqref{eq:kld2}, we only need to show that 
$$
\E^{\overline{\nu}_\eps}\int_0^1 \Psi_\eps(z(t) + m^\ast(t))dt < \infty.
$$
Since $\bA^\ast = a\cdot \bI_d$, from \eqref{eq:greenexp} one can see that $z(t) \sim N(0, 2\bG_\eps(t,t))$ under the measure $\overline{\nu}_\eps$. In addition, it follows from \eqref{eq:greenF} that 
$|\bG_\eps(t,t)|_F \leq C\eps$ a.e. on $(0,1)$ for some $C> 0$. 
Then from the growth condition (A-4) on $\Psi_\eps$ and the fact that $m^\ast\in \bL^\infty(0,1)$,
$$
\bea
& \E^{\overline{\nu}_\eps}\int_0^1 \Psi_\eps(z(t) + m^\ast(t))dt \\
& =
\int_0^1 \int_{\R^d} \frac{1}{\sqrt{(4\pi)^d \text{det}(\bG_\eps(t,t))}} e^{-\frac{1}{4}x^T\bG_\eps(t,t)^{-1}x} \Psi_\eps(x + m^\ast(t)) dx dt\\
& = \int_0^1 \int_{\R^d} \frac{1}{(4\pi)^{d/2}} e^{-\frac{1}{4}|x|^2 } \Psi_\eps\left((\bG_\eps(t,t)^{1/2}) x + m^\ast(t)\right) dx dt\\
& \leq C_1 \text{exp}\left(\|m^\ast\|^\alpha_{\bL^\infty(0,1)}\right) \int_{\R^d} e^{-\frac{1}{2} |x|^2 + C_2 \eps^\alpha |x|^\alpha} dx <  \infty
\ena
$$
since $\alpha \in  [0, 2)$.

Next, we prove that the minimizer exists. By examining the proof of \cite[Theorem 3.10]{PSSW15a}, one can see that the theorem is proved if the following statement is valid: 
if a sequence $\{\bA_n\}\subset \bH^1_a(0,1)$ satisfies $\sup_n \|\bA_n\|_{\bH^1(0,1)} < \infty$, then the sequence $\{\bB_n\}$ with $B_n = \eps^{-2} \bA_n^2 - \eps^{-1}\bA_n^{\prime}$, viewed as multiplication operators, contains a subsequence that converges to $\bB = \eps^{-2}\bA^2 - \eps^{-1}\bA^\prime$ in $\mathcal{L}(\bH^{\beta}, \bH^{-\beta})$ for some $\bA\in \bH^1_a(0,1)$ and some $\beta \in (0, 1)$. Hence we only need to show that the latter statement is true. In fact, if $\sup_n \|\bA_n\|_{\bH^1(0,1)} < \infty$, then there exists a subsequence $\{\bA_{n_k}\}$ and some $\bA\in \bH^1(0,1)$ such that $\bA_{n_k} \wgt \bA$ in $\bH^1(0,1)$. By Rellich's compact embedding theorem, $\bA_{n_k} \gt \bA$ in $\bL^2(0,1)$ and passing to a further subsequence we may assume that $\bA_{n_k} \gt \bA$ a.e. on $[0,1]$. This implies that $\bA$ is symmetric and $\bA \geq a\cdot \bI_d$ a.e. and hence $\bA\in \bH^1_a(0,1)$. In addition, it is clear that $\bB_{n_k} \wgt \bB$ in $\bL^2(0,1)$. According to Lemma \ref{lem:app1}, for any $\alpha,\beta > 0$ such that $\beta >\max(\alpha, \alpha/2 + 1/4)$, a matrix-valued function in $\bH^{-\alpha}(0,1)$ can be viewed as a multiplication operator in $\mathcal{L}(\bH^\beta, \bH^{-\beta})$. Thanks to the compact embedding from $\bL^2(0,1)$ to $\bH^{-\alpha}(0,1)$, we obtain $\bB_{n_k} \gt \bB$ in $\mathcal{L}(\bH^\beta, \bH^{-\beta})$. The proof is complete. 
\end{proof}


\begin{rem}
Minimizers of \eqref{eq:minifunc} are not unique in general. The uniqueness issue is outside the scope of this paper; see more discussions about uniqueness of minimizing the Kullback-Leibler divergence in \cite[Section 3.4]{PSSW15a}. $\qed$
\end{rem}


\section{Low Temperature Limit}\label{sec:gammalim}
In this section, we aim to understand the low temperature limit of the best Gaussian approximations discussed in the previous section. This will be done in the framework of $\Gamma$-convergence. First we recall the definition of $\Gamma$-convergence (see \cite{bra02a,D93}) and introduce some functionals which are closely related to the Gaussian approximations.
\subsection{Notion of $\Gamma$-Convergence and Preliminaries}
\begin{defn}
Let $\mathcal{X}$ be a topological space, $\eps > 0$ and $F_\eps : \mathcal{X} \gt \R$ a family of functionals. We say that $F_\eps$ $\Gamma$-converges to $F:\mathcal{X}\gt \R$ as $\eps \gt 0$ if the following two conditions hold:

(i) (Liminf inequality) for every $u\in \mathcal{X}$ and every sequence $u_\eps\in \mathcal{X}$ such that $u_\eps \gt u$,

\ben
F(u) \leq \liminf_{\eps \gt 0}  F_{\eps} (u_{\eps});
\enn

(ii) (Limsup inequality) for every $u\in \mathcal{X}$ there exists a sequence $u_\eps$ such that $u_\eps \gt u$ and 
\ben
F(u) \geq \limsup_{\eps \gt 0} F_\eps(u_\eps).
\enn
\end{defn}
For studying the low temperature limit of the Gaussian approximations, 
we consider the following family of functionals: 
 \be\label{eq:Fd-0}
F_\eps(m, \bA) := \begin{cases}
\eps \widetilde{D}_{\text{KL}}(\nu_\eps || \mu_\eps) + \eps^{\gamma}\|\bA\|^2_{\bH^1(0,1)}, & \text{ if } (m, \bA) \in \HH,\\
\infty, & \text{otherwise in } \mathcal{X}
\end{cases}
\en
on the space $\mathcal{X}  = \bL^1(0,1) \times \bL^1(0,1)$.
Then minimizing \eqref{eq:minifunc} over $\A$ is equivalent to the following problem
\be\label{eq:min1}
\inf_{(m, \bA)\in \mathcal{X}} F_\eps(m, \bA).
\en
In order to study the $\Gamma$-limit of $F_\eps$,
 we equip the space $\mathcal{X}$ with a product topology 
such that the convergence $(m_\eps, \bA_\eps) \gt (m, \bA)$ in $\mathcal{X}$ means that $m_\eps \gt m$ in $\bL^1(0,1)$ and that $\bA_\eps \wgt \bA$ in $\bL^1(0, 1)$. The reason for choosing the weak topology for $\bA$ is that the functional $F_\eps$ is coercive under such topology only, see Proposition \ref{prop:compact-2}. 
Now before we proceed to discussing the $\Gamma$-convergence of $F_\eps$, we first state a useful $\Gamma$-convergence result for the 
classical Ginzburg-Landau functional 
\be\label{eq:Eeps} 
E_\eps (m):= \begin{cases}
\frac{\eps}{4} \int_0^1 |m^\prime(t)|^2 dt + \frac{1}{4\eps} \int_0^1 |\nabla V(m(t))|^2 dt & \text{ if } m\in \bH^1_{\pm}(0,1),\\
\infty, \text{ otherwise in } \bL^1(0,1).
\end{cases}
\en
Notice that in the above definition, any $m$ such that $E_\eps(m)$ is finite should satisfy the Dirichlet boundary conditions $m(0) = x_-$ and $m(1) = x_+$. 
We also remark that after performing the scaling transformation $t \mapsto \eps^{-1} t$,  the functional $E_\eps$ coincides with the Freidlin-Wentzell functional $S_T$ (defined in \eqref{eq:FW}) with $T = \eps^{-1}$. Indeed, by rewriting $\widetilde{m}(\cdot) = m(\eps^{-1}\cdot)$, one sees that $E_\eps(m) = S_{T}(\widetilde{m})$.

To define the $\Gamma$-limit of $E_\eps$, we now introduce some additional notations. Recall that $\EE$ defined in \eqref{eq:criticalset} is the set of critical points of V.
For each pair $x_-, x_+\in \EE$,  we define the set of transition paths  
$$\bX(x_-, x_+) := \{m\in \BV(\R)\ |\ \lim_{t\gt \pm \infty} m(t) = x_\pm \text{ and } m^\prime \in \bL^2(\R)\},$$
 the cost functional 
\be\label{eq:JT}
\mathcal{J}_{T}(m) = \frac{1}{4} \int_{-T}^{T} \Bigl(|m^\prime(t)|^2 + |\nabla V (m(t))|^2 \Bigr)dt,
\en
and set $\mathcal{J}(m) := \mathcal{J}_{\infty}(m)$. 
The minimal transition cost from $x_-$ to $x_+$ is  then defined as
$$\Phi(x_-, x_+) := \inf \{\mathcal{J}(m)\ |\ m \in \bX(x_-, x_+)\}.$$ 
It is worth noting that the function $\Phi(x_-, x_+)$ is closely related to the so-called quasi-potential, which plays an important role in large deviation theory: In fact, suppose that $x_-, x_+ \in \mathcal{E}$ and that $V$ satisfies Assumption \eqref{assump}. Then according to \cite[Lemma 3.2]{FT89}, the function $\Phi(x_-, x_+)$ has the following equivalent form:
\be\label{eq:eqphi}
\bea
 &\Phi(x_-, x_+) \\
 & = \inf_{T, m} \Big\{ \mathcal{J}_T(m) : T > 0, m\in \bH^1(-T, T)  \text{ and } m(-T) = x_-, m(T) = x_+ \Big\}.
\ena
\en
This definition shows that $ \Phi(x_-, x_+)$ coincides with the quasi-potential between $x_-$ and $x_+$ (as defined in \cite[Chapter 4]{FSW12})
up to the additive constant $- \frac{1}{2} (V(x_+) - V(x_-))$.

We also remark that the equivalent formulation \eqref{eq:eqphi} provides an important ingredient for proving the $\Gamma$-convergence of $E_\eps$; see e.g. \cite{FT89,bra02a}. Given $x_\pm \in \EE$, if either $x_-$ or $x_+$ is a local minimum or maximum of potential $V$ and if $V$ satisfies (A-1)-(A-3) of Assumption \ref{assump}, it was shown in \cite[Lemma 2.1]{PST} that the infimum $\Phi(x_-, x_+)$ is attained by the heteroclinic orbits $m_\ast$ of the Hamiltonian system 
$$
m_\ast^{\dpm}(t) - D^2 V(m_\ast) \nabla V(m_\ast) = 0, \quad \lim_{t \gt \pm \infty} m(t) = x_\pm.
$$
In this case,  
\be\label{eq:minimumcost}
\Phi(x_-, x_+) = \frac{1}{2} |V(x_+) - V(x_-)|.
\en

Denote by $\BV(0,1; \EE)$ the set of functions in $\BV(0,1)$ taking values in $\EE$ a.e. on $[0,1]$. For any $u\in \BV(0,1; \EE)$, let $J(u)$ be the set of jump points of $u$ on $(0,1)$, and let $u(t^\pm)$ the left and right sided limits of $u$ at time $t\in [0,1]$. 
The following lemma, concerning the compactness of $E_\eps$, will be very useful in identifying its $\Gamma$-limit. Its proof can be found in \cite[Theorem 1.2]{L14}. 
\begin{lem}\label{lem:compact-1}
Assume that the potential $V$ satisfies (A-1)-(A-3). Let $\eps_n\gt 0$ and let $\{m_{n}\} \subset \bH^1_{\pm}(0,1)$ be such that
$$
\limsup_{n\gt \infty} E_{\eps_n}(m_{n}) < \infty.
$$
Then there exists a subsequence $\{m_{n_k}\}$ of  $\{m_{n_k}\}$ and an $m\in \BV(0,1; \EE)$ such that 
$m_{n_k} \gt m$ in $\bL^1(0,1)$ as $k \gt \infty$.
\end{lem}
We remark that we incorporate the boundary conditions $m_n(0)= x_-, m_n(1) = x_+$ in the statement of the lemma since $m_n \in \bH^1_{\pm}(0,1)$.
The following Proposition identifies the $\Gamma$-limit of $E_\eps$ with respect to $\bL^1$-topology; this is based upon Lemma \ref{lem:compact-1} and the standard Modica-Mortola type arguments (see \cite{MM77,Bal90,PST}). The proof is given in Appendix \ref{app-d}. The same $\Gamma$-convergence result was claimed in \cite{PST}, but the proof there was actually carried out with respect to the topology in the space of functions of bounded variations.
\begin{prop}\label{prop:gammalimit-E}
Assume that $V$ satisfies the conditions (A-1)-(A-3), the $\Gamma$-limit of $E_\eps$ is 
\be\label{eq:Em}
\bea
E (m) := \begin{cases}
\Phi(x_-, m(0^+))  +\sum_{\tau\in J(m)} \Phi(m(\tau^-), m(\tau^+))\\
\qquad  + \Phi(m(1^-), x_+) & \text{ if } m\in \BV(0,1;\EE),\\
+\infty & \text{ otherwise in } \bL^1(0,1).
\end{cases}
\ena
\en
\end{prop}

\subsection{Main Results}
 This subsection presents the main results about the $\Gamma$-convergence of the functional $F_\eps$; the proofs will be presented in the next section. 
Roughly speaking, our arguments indicate that the $\Gamma$-limit of $F_\eps$ on $\mathcal{X}$ should be 
\begin{equation}\label{e:main-functional-1}
F(m, \bA) := E(m) + \frac{1}{4} \int_0^1 \left(D^2 V(m(t)) - |\bA(t)|\right)^2 : |\bA^{-1}(t)| dt 
\end{equation}
where $E(m)$ is defined by \eqref{eq:Em}. Recall that $\bA : \bB = \Tr(\bA \bB^T)$. However, for technical reasons, we are only able to prove the claim under the condition that the matrix $\bA$ is positive definite; see Remark \ref{rem:positiveA}. To make this clear, let us first redefine $F_\eps$ to be
 \be\label{eq:Fd}
F_\eps(m, \bA) := \begin{cases}
\eps \widetilde{D}_{\text{KL}}(\nu_\eps || \mu_\eps) + \eps^{\gamma}\|\bA\|^2_{\bH^1(0,1)}, & \text{ if } (m, \bA) \in \HH_a,\\
\infty, & \text{otherwise in } \mathcal{X}_a
\end{cases}
\en
with some $a > 0$.
Then we can show that $F_\eps$ as defined in \eqref{eq:Fd} $\Gamma$-converges to $F$ defined by \eqref{e:main-functional-1} on the space $\mathcal{X}_a$ for any $a > 0$; see Theorem \ref{thm:main}. Recall that $\mathcal{X}_a = \bL^1(0,1)\times \bL^1_a(0,1)$ and that convergence  of $(m_n, \bA_n)$ in $\mathcal{X}_a$ means that the $m_n$ converge strongly in $\bL^1(0,1)$ and the $\bA_n$ converge weakly in $\bL^1_a(0,1)$.

By the definition of $F_\eps$ (by \eqref{eq:Fd}) and the expression \eqref{eq:kld} for $\widetilde{D}_{\text{KL}}(\nu_\eps || \mu_\eps)$, we can write 
\be\label{eq:feps}
F_\eps (m, \bA) = F_\eps^{(1)}(m, \bA) + F_\eps^{(2)}(\bA) + \eps^\gamma \|\bA\|_{\bH^1(0,1)}^2 
\en
for $(m, \bA)\in \HH_a$ where 
\be\label{eq:f1}
\begin{aligned}
 F_\eps^{(1)}(m, \bA) &:= \frac{\eps}{4}\int_0^1 |m^\prime (t)|^2 dt + \frac{1}{2\eps} \E^{\overline{\nu}_\eps} \left[\int_0^1 \Psi_\eps(z(t) + m(t)) dt\right],
\\
F_\eps^{(2)}(\bA) &:= - \frac{\eps}{4} \E^{\overline{\nu}_\eps} \left[\int_0^1 z(t)^T \bB_\eps(t) z(t) dt\right]+ \frac{1}{2}\int_0^1 \Tr(\bA(t))dt\\
& + \frac{\eps}{2}\log\left(\det\left(\int_0^1 \overline{\bM}_\eps(t)\overline{\bM}_\eps(t)^T dt\right)\right),
 \end{aligned}
 \en
 where $\Psi_\eps$ is given by \eqref{eq:psi} and $\overline{\bM}_\eps$ is defined by \eqref{eq:M}. 
 To identify the $\Gamma$-limit of $F_\eps$, we need to study the liminf or limsup of the sequence $\{F_\eps(m_\eps, \bA_\eps)\}$ with $m_\eps\in \bH^1_{\pm}(0,1)$ and $\bA_\eps \in \bH^1_a(0,1)$. This is non-trivial in our case, mainly because the functional $F_\eps$ depends on $m$ and $\bA$ in 
an implicit manner through the two expectation terms. Therefore in the first step we shall simplify $F_\eps$. The following proposition examines the limiting behavior of the functional $F_\eps$ from which a simplified and more explicit expression is obtained.

\begin{prop}\label{prop:sfeps}
Let $(m_\eps, \bA_\eps) \in \HH_a$. Assume that for some $\gamma\in (0, \frac{1}{2})$, 
$$\limsup_{\eps \gt 0} \eps^\gamma \|\bA_\eps\|^2_{\bH^1(0,1)} < \infty\quad \text{ and }\quad \limsup_{\eps \gt 0} \|m_\eps\|_{\bL^\infty(0,1)} < \infty.$$
 Then for $\eps>0$ small enough we have 
\be\label{eq:fepsk}
\bea
F_{\eps} (m_{\eps}, \bA_{\eps}) 
& = E_{\eps}(m_{\eps}) + \frac{1}{4}\int_0^1 \left(D^2 V(m_{\eps}(t)) - \bA_{\eps}(t)\right)^2: \bA_{\eps}^{-1}(t) dt\\
& + \int_0^1  \big( D^3 V(m_{\eps}(t))\cdot \nabla V(m_\eps(t)) \big): \bA_{\eps}^{-1}(t) dt + \eps^\gamma \|\bA_{\eps}\|_{\bH^1(0,1)}^2 + \cO(\eps^{\frac{1}{2}}).
\ena
\en
\end{prop}
The proof of Proposition \ref{prop:sfeps} requires several technical lemmas and is referred to Section \ref{subsec:proofs}.
The basic idea for proving Proposition \ref{prop:sfeps} is as follows. First one can express the expectation term in $F_\eps^{(2)}(\bA_\eps)$ in terms of the Dirichlet Green's tensor of some Schr\"odinger operator (see \eqref{eq:f_2expterm}). A careful asymptotic analysis of this Green's tensor implies  that 
\be\label{eq:f2epsasym}
F_\eps^{(2)}( \bA_\eps)  \approx \frac{1}{4}\int_0^1 \Tr(\bA_\eps(t)) dt,
\en
see Corollary \ref{cor:expt} for the precise statement. For the expectation term in $F_\eps^{(1)}(m_\eps, \bA_\eps)$, we approximate $\Psi_\eps(X)$ by its second order Taylor expansion around the mean $m_\eps$. The zero order term of the expansion is $\Psi_\eps(m_\eps) = \frac{1}{2} |\nabla V(m_\eps)|^2 - \eps \Delta V(m_\eps)$. Then $E_\eps(m_\eps)$ is obtained by combining the  term $\frac{\eps}{4}\int_0^1 |m^\prime (t)|^2 dt$  in $F_\eps^{(1)}(m_\eps, A_\eps)$ with the integral over $\frac{1}{4\eps} |\nabla V(m_
\eps)|^2$. Additionally, the It\^o correction term $-\eps \Delta V(m_\eps)$, which is the other zero order term of the Taylor expansion, can be combined with one of the second order terms of the expansion and \eqref{eq:f2epsasym} to complete the full quadratic term in \eqref{eq:fepsk}.

As a consequence of Proposition \ref{prop:sfeps}, we get the following interesting compactness result for the functional $F_\eps$. 
\begin{prop}\label{prop:compact-2}
 Let $\eps_n \gt 0$ and let $\{(m_n, \bA_n)\}$ be a sequence in $\HH_a$ such that
  $$\limsup_n F_{\eps_n} (m_n, \bA_n) < \infty.$$
   Then there exists a subsequence $\{(m_{n_k}, \bA_{n_k})\}$ of $\{(m_n, \bA_n)\}$ such that $m_{n_k} \gt m$ in $\bL^1(0,1)$ and $\bA_{n_k} \wgt \bA$ in $\bL^1(0,1)$ with $m\in \BV(0,1; \EE)$ and $\bA \in \bL_a^1(0,1)$. 
\end{prop}
%
This compactness result is slightly weaker than the usual compactness property relevant to $\Gamma$-convergence (see e.g. the conclusion in Lemma \ref{lem:compact-1}), because only weak convergence is obtained for the variable $\bA$. 
%
Building upon the $\Gamma$-convergence result of $E_\eps$, Proposition \ref{prop:sfeps} and Proposition \ref{prop:compact-2}, the following main theorem establishes the $\Gamma$-convergence of $F_\eps$.

\begin{thm}\label{thm:main}
Suppose that $V$ satisfies the assumptions (A-1)-(A-6). Let $\gamma \in (0, \frac{1}{2})$ in \eqref{eq:feps}. Then the $\Gamma$-limit of $F_\eps$ defined by \eqref{eq:Fd} on $\mathcal{X}_a$  is 
\begin{equation}\label{e:main-functional}
F(m, \bA) = E(m) + \frac{1}{4} \int_0^1 (D^2 V(m(t)) - \bA(t))^2 : \bA^{-1}(t) dt 
\end{equation}
where $E(m)$ is defined by \eqref{eq:Em}. 
\end{thm}

$\Gamma$-convergence of $F_\eps$ implies convergence of minima. 
\begin{cor}\label{cor:minima}
Let $(m_\eps, \bA_\eps) \in \HH_a$ be minimizes of $F_\eps$. Then up to extracting a subsequence,  $m_\eps \gt m$ in $\bL^1(0,1)$ and $\bA_\eps \wgt \bA$ in $\bL^1(0,1)$ for some $m\in \BV(0,1;\EE)$ and $\bA \in \bL_a^1(0,1)$. Furthermore, the limit $(m, \bA)$ is a minimizer of $F$ on $\mathcal{X}_a$. 
\end{cor}

\begin{rem}
In general convergence of minima requires both (strong) compactness and $\Gamma$-convergence; see e.g. \cite{bra02a}. In our case we only have weak compactness with respect to $\bA_\eps$ for $F_\eps$; see Proposition \ref{prop:compact-2}. However, such weak convergence of $\bA_\eps$ suffices to pass to the limit because the leading order term of the functional $F_\eps(m_\eps, \bA_\eps)$ is convex with respect to $\bA_\eps$. See the analysis of the functional \eqref{eq:functionalM} in the next section. $\qed$
\end{rem}


\begin{rem}\label{rem:positiveA}
Theorem \ref{thm:main} shows the $\Gamma$-convergence of $F_\eps$ to $F$ (given by \eqref{e:main-functional}) under the assumption that $\bA$ is bounded away from zero, i.e. $\bA(\cdot) \geq a\cdot \bI_d$ for some $a > 0$. However, this assumption is unlikely to be sharp. In fact, under the weaker positivity assumption that $|\bA(\cdot)| \geq a\cdot \bI_d$, one can at least prove the liminf part of the $\Gamma$-convergence of $F_\eps$ to $F$ defined in \eqref{e:main-functional-1}. This is mainly because the leading order of the Green's function $\bG_\eps$ (defined by \eqref{eq:grt}) depends only on $|\bA|$; see \eqref{eq:grt-4} of Lemma \eqref{lem:a2}. Although the positivity assumption is essential in our arguments for proving Theorem \eqref{thm:main}, we conjecture that the $\Gamma$-convergence result is still valid without any positivity assumption. This is to be investigated in future work.
\end{rem}

\section{Conclusion}\label{sec:conclusion}


The Freidlin-Wentzell theory gives a quantitative description of the tail of the distribution of transition paths based on the theory of large deviations. It thereby leads to a natural variational definition of most likely paths in the low temperature limit, namely the minimizers of the large deviation rate functional. However, this approach exhibits some weaknesses. In particular the large deviation theory 
of Freidlin-Wentzell makes asymptotic statements in the limit where the noise 
intensity goes to zero. In practical applications the noise level may not be small 
enough for these asymptotics to be valid. Furthermore, the large deviation rate 
functional does not exclude the possibility of the path spending large stretches of 
time near local maxima or saddles of the potential $V$. 
The Onsager-Machlup theory offers an alternative variational definition of 
most likely paths at finite temperature in terms of MAP estimators, but as shown in \cite{PS} minimizers of the Onsager-Machlup functional may be unphysical at
small temperatures because the methodology fails to account for entropic effects; this
can lead to transition paths which choose to make transitions through
narrow energy barriers rather than (entropically favourable) wider ones
with the same height, or to transition paths which (like Freidlin-Wentzell paths,
although for different reasons) can spend long times at a saddle point.

We have developed an approach to the problem of identifying the most
likely transition path which (like the Onsager-Machlup approach) is
well-defined at finite non-vanishing temperature and yet which also recovers the
correct limiting behaviour in the small temperature limit
(minimizers of the Freidlin-Wentzell least action principle). Furthermore,
our approach is based on finding the best Gaussian approximation with respect 
to Kullback-Leibler divergence, and hence captures  not only the
most likely path, but also the fluctuations around it. In the
small temperature limit this gives the appealing interpretation that
the fluctuations are defined by an OU process found from linearizing the
Brownian dynamics model at the minimizer of the Freidlin-Wentzell action.
It is thus important to recongnize that our work leads to useful 
characterizations of transition paths in the Brownian dynamics model,
both at finite $\eps$ and in the limit $\eps \to 0$. In this paper
we have concentrated exclusively  on the $\eps \to 0$ limit. However
we now make some remarks that have bearing on both of these parameters
regimes.

\subsection{Computational Methods (Fixed $\eps$)}

Even though there is no explicit {\em analytic} characterization 
for our notion of most likely path at finite temperature, 
 it is possible to calculate it {\em numerically}. In fact, for a fixed finite temperature $\eps$, finding the best Gaussian approximation requires minimizing the functional $F_\eps$, which involves the Kullback Leibler divergence. In \cite{PSSW16}, a variant of the Robbins-Monro algorithm has been introduced
to find best Gaussian approximations with respect to
the Kullback Leibler divergence; the numerical results in that paper
demonstrate the feasability of the minmization, and also demonstrate that
the resulting Gaussian approximation can be used to construct
improved MCMC algorithms for transition paths sampling. 

An important aspect of any gradient descent method to minimize
an objective function is the initialization. When the temperature $\eps$
is small but finite, the analysis in this paper also suggests a good
initialization. Proposition \ref{prop:sfeps} gives an approximate formula for the functional $F_\eps$:
$$
F_\eps = \overline{F}_\eps(m_\eps, \bA_\eps) + o(1).
$$
where 
$$
\overline{F}_\eps(m, \bA):= E_\eps (m) + \frac{1}{4} \int_0^1 (D^2 V(m(t)) - \bA(t))^2 : \bA^{-1}(t) dt.
$$
As a consequence, a minimizer of $\overline{F}_\eps$ provides a good approximation for the minimizer of $F_\eps$. Approximate minimization of the former, to obtain
a good initialization, can be carried out by make several alternations of
the following two steps: first, freeze $\bA$ and minimize $\overline{F}_\eps(m, \bA)$ as a functional of $m$ -- this can be implemented via many minimum action algorithms (see e.g. \cite{ERV04, HV08, GSV16}); second, with a good approximation for $m$ being frozen, update $\bA$ to be the minimizer of 
$$
\int_0^1 (D^2 V(m(t)) - \bA(t))^2 : \bA^{-1}(t) dt,
$$
which gives $\bA(t) = |D^2 V(m(t))|$ (see Lemma \ref{lem:secondofF}).

\subsection{{Interpretation of Small $\eps$ Analysis}}

The discussion about initializing the numerical minimization for
small $\eps$ also helps to explain our earlier assertions about
the desirable structure of our minimizers when $\eps$ is small.
Our main result, Theorem~\ref{thm:main}, shows that in the small temperature limit, the KL-minimization improves on the predictions obtained by minimizing the large deviation rate functional; the $\Gamma$-limit of the functional $F_\eps$ given in \eqref{e:main-functional-1} consists of two parts whose minimization decouples
in the limit $\eps=0$. The first part $E$ 
is closely linked with large deviation theory since it is the $\Gamma$-limit of the scaled Freidlin-Wentzell functional $E_\eps$ (Recall that $E_\eps$ arises from the Freidlin-Wentzell functional $\overline{S}_T$, defined in \eqref{eq:action}, by scaling $T = \eps^{-1}$  and removing the constant  $\frac{1}{2}(V(x_+) - V(x_-))$).
 The second part keeps track of the Ornstein-Uhlenbeck fluctuations around the optimal path and 
 thereby captures entropic effects. Given a path $m$ it can be minimized by choosing $\bA(t) =| D^2 V(m(t))|$. 
In particular, for this choice $(D^2 V(m(t)) - |\bA(t)|)^2 : |\bA^{-1}(t)| $ is equal to zero if $D^2 V(m(t))$ is positive definite (corresponding to $m(t)$ being a local minimizer of $V$) and strictly positive if $D^2 V(m(t))$ has a negative eigenvalue (corresponding to $m(t)$ being a saddle point or a local maximum). 
 Therefore, the $\Gamma$-limit $F$ can be minimized explicitly as follows: first find a minimizer of $E$ which amounts to 
 selecting the sequence of critical points connecting $x_-$ and $x_+$ that minimizes the transition cost (defined by \eqref{eq:JT}). As shown in Section \ref{sec:gammalim} the minimal transition cost equals to the Freidlin-Wentzell quasi-potential up to a constant and in simple cases it is given by \eqref{eq:minimumcost}. Then in order 
 to minimize also the entropic part, select among the paths which follow this sequence, those that spend no time in saddles or local maximisers.
This also shows that the unphysical minimizers of the Onsager-Machlup approach (discussed in the introduction section) are removed in our approach.

\section{Proofs of Main Results}\label{sec:proofs}

 \subsection{Asymptotics of $F^{(2)}_\eps(\bA_\eps)$}
%

Let $\bG_\eps(t, s)$ be the Green's tensor (fundamental matrix) of the elliptic operator $(-\bD + \bB_\eps)$ under Dirichlet boundary conditions, i.e. for any $s \in (0,1)$,
 \be\label{eq:grt-}
 \bea
& \left(-\bD + \eps^{-2}\bA_\eps^2(\cdot) - \eps^{-1} \bA_\eps^\prime(\cdot)\right) \bG_\eps(\cdot, s) = \delta(\cdot - s)\cdot\bI_d,\\
& \bG_\eps(0, s) = \bG_\eps(1, s) =  0.
\ena
 \en
 Then by the definition of covariance operator, the expectation term in $F_\eps^{(2)}(\bA_\eps)$ can be calculated in terms of the Green's tensor $\bG_\eps$. More precisely, 
 \be\label{eq:f_2expterm}
- \frac{\eps}{4} \E^{\overline{\nu}_\eps} \left[\int_0^1 z(t)^T \bB_\eps(t) z(t) dt\right] = -\frac{\eps}{2}\int_0^1 \bB_\eps(t): \bG_\eps(t, t) dt.
 \en
 To simplify $F_\eps^{(2)}$ we need the asymptotic estimates of $\bG_\eps$ for small $\eps$, which we show in the following. 

\subsubsection{Asymptotic Estimates of The Green's Tensor}\label{sec:greentensor}
For fixed $s\in (0,1)$, the Green's tensor $\bG_\eps(\cdot, s)$ solves the linear elliptic PDE system \eqref{eq:grt-} with variable coefficient. We want to approximate $\bG_\eps$ by a simple Green's tensor, for which an explicit asymptotic formula is available. To do this, for any $s\in (0,1)$, we define $\overline{\bG}_\eps(\cdot, s)$ such that
\be\label{eq:grt-1}
\bea
& \left(-\bD + \eps^{-2} \bA_\eps^2(s)\right) \overline{\bG}_\eps(\cdot, s) = \delta(\cdot - s) \cdot \bI_d, \\
& \overline{\bG}_\eps(0, s) = \overline{\bG}_\eps(1, s) = 0.
\ena
\en
According to Lemma \ref{lem:a2}, when $\eps$ is small
\be\label{eq:overgr}
\overline{\bG}_\eps(t,t) = \frac{\eps}{2} (\bA_\eps^{-1}(t) + \bR_\eps(t))
\en with 
$
|\bR_\eps(t)| \leq C\big(e^{-\frac{2at}{\eps}} +  e^{-\frac{2a(1 - t)}{\eps}}\big).
$
Remember that $a$ is the constant for which we have  $\bA_\eps(t) \geq a\cdot \bI_d$  a.e. by assumption.
Furthermore, the difference $\widetilde{\bR}_\eps(t,s) = \bG_\eps(t, s) - \overline{\bG}_\eps(t, s)$ admits the following bound for small $\eps$.
\begin{lem}\label{lem:grt}
Let $\gamma \in (0, \frac{1}{2})$ in \eqref{eq:Fd}. Let $\{\bA_\eps\} \subset \bH^1_a(0,1)$ such that $$\limsup_{\eps \gt 0} \eps^\gamma \|\bA_\eps\|^2_{\bH^1_a(0,1)}  < \infty.$$  Then for $\eps$ sufficiently small we have that
\be\label{eq:grtinfty}
\sup_{s\in (0,1)} \|\widetilde{\bR}_\eps(\cdot,s)\|_{\bL^\infty(0,1)} \lesssim \eps^{\frac{3}{2} - \gamma},
\en
and that
\be\label{eq:grt2}
\sup_{s\in (0,1)} \|\widetilde{\bR}_\eps(\cdot,s)\|_{\bL^2(0,1)} \lesssim \eps^{2 - \gamma}.
\en
\end{lem}
\begin{proof}
According to \eqref{eq:grt-} and \eqref{eq:grt-1}, $\widetilde{\bR}_\eps$ satisfies
$$
\bea
& \left(-\bD  + \eps^{-2} \bA_\eps^2(t) - \eps^{-1} \bA_\eps^\prime(t) \right)\widetilde{\bR}_\eps(t, s) = \bF_\eps(t, s),\\
& \widetilde{\bR}_\eps(0, s) = \widetilde{\bR}_\eps(1, s) = 0,
\ena
$$
with 
$$
\bF_\eps(t, s):= \left(\eps^{-2}(\bA_\eps^2(s) - \bA_\eps^2(t)) + \eps^{-1}\bA_\eps^\prime(t)\right) \overline{\bG}_\eps(t, s).
$$
Let $\widetilde{R}_\eps^i, \overline{G}_\eps^i, F_\eps^i$ be the $i$-th column of the matrices $\widetilde{\bR}_\eps, \overline{\bG}_\eps, \bF_\eps$ respectively.
\be\label{eq:grt-i}
\bea
& \left(-\bD  + \eps^{-2} \bA_\eps^2(t) - \eps^{-1} \bA_\eps^\prime(t) \right)\widetilde{R}_\eps^i(t, s) = F_\eps^i(t, s),\\
& \widetilde{R}_\eps^i(0, s) = \widetilde{R}_\eps^i(1, s) = 0.
\ena
\en
We only need to prove estimates \eqref{eq:grtinfty} and \eqref{eq:grt2} for each column $\widetilde{R}_\eps^i, i=1,\cdots, d$. To this end, we first bound the $\bL^1$-norm of the right hand side $F_\eps^i$. In fact, by Morrey's inequality (see e.g. \cite[Chapter 5]{E2010}), it holds that
$$
|\bA_\eps(t) - \bA_\eps(s)|_F \lesssim \|\bA_\eps^\prime\|_{\bL^2(0,1)}\cdot |t - s|^{\frac{1}{2}}
$$
for any $t,s \in [0,1]$. This together with \eqref{eq:grt-3} implies that
\be\label{eq:fepsl1}
\bea
& \|F_\eps^i(\cdot, s)\|_{\bL^1(0,1)} \\
&\leq 2\eps^{-2}\|\bA_\eps\|_{\bL^\infty(0,1)} \int_0^1 |(\bA_\eps(s) - \bA_\eps(t))\overline{G}_\eps^i(t, s)| dt 
 + \eps^{-1}\int_0^1 |\bA_\eps^\prime(t) \overline{G}_\eps^i(t, s)| dt\\
& \lesssim \eps^{-1}\|\bA_\eps\|_{\bL^\infty(0,1)} \cdot \|\bA_\eps^\prime\|_{\bL^2(0,1)}\cdot \int_0^1 |t - s|^{\frac{1}{2}} e^{-\frac{a|t- s|}{\eps}} dt\\
& \qquad + \|\bA_\eps^\prime\|_{\bL^2(0,1)}\cdot \|e^{-\frac{a|\cdot - s|}{\eps}}\|_{L^2(0,1)}\\
& \lesssim \eps^{\frac{1}{2}}\left(\|\bA_\eps\|_{\bL^\infty(0,1)} + 1\right)\|\bA_\eps^\prime\|_{\bL^2(0,1)} \\
& \lesssim \eps^{\frac{1}{2}}\left(\|\bA_\eps\|_{\bH^1(0,1)} + 1\right) \|\bA_\eps\|_{\bH^1(0,1)} \lesssim \eps^{\frac{1}{2} - \gamma}
\ena 
\en
where we have used the Sobolev embedding $\bH^1(0,1) \hookrightarrow \bL^\infty(0,1)$ and the assumption that $\eps^\gamma \|\bA_\eps\|^2_{\bH^1(0,1)} < \infty$ in the last two inequalities. Now taking the dot product of the equation \eqref{eq:grt-i} and $\widetilde{R}^i_\eps(\cdot,s)$ and integrating over $(0,1)$, one obtains that
\be\label{eq:grt-a}
\bea
|\widetilde{R}_\eps^i(\cdot,s) |_{\bH^1(0,1)}^2 + \frac{a^2}{\eps^2} \|\widetilde{R}_\eps^i(\cdot,s)\|_{\bL^2(0,1)}^2 & \leq \eps^{-1}\int_0^1| (\widetilde{R}_\eps^i(t,s))^T \bA_\eps^\prime(t) \widetilde{R}_\eps^i(t,s)|  dt \\
& + \int_0^1 |\bF_\eps(t, s)\cdot \widetilde{R}_\eps^i(t,s)| dt.
\ena
\en
We claim that the first term on the right side can be neglected when $\eps$ is small. In fact, using the Sobolev embedding $\bH^{\frac{1}{4}}(0,1)\hookrightarrow \bL^4(0,1)$ and the interpolation inequality of Lemma \eqref{lem:int}, we obtain that 
$$
\bea
&  \eps^{-1}\int_0^1| (\widetilde{R}_\eps^i(t,s))^T \bA_\eps^\prime(t)\widetilde{R}_\eps^i(t,s)|  dt \leq \eps^{-1}
\|\bA_\eps^\prime\|_{\bL^2(0,1)}\|\widetilde{R}_\eps^i(\cdot,s)\|^2_{\bL^4(0,1)} \\
& \leq C\eps^{-1} \|\bA_\eps^\prime\|_{\bL^2(0,1)}\|\widetilde{R}_\eps^i(\cdot,s)\|^2_{\bH^{\frac{1}{4}}(0,1)} \\
& \leq C \eps^{-1}\|\bA_\eps^\prime\|_{\bL^2(0,1)}\|\widetilde{R}_\eps^i(\cdot,s)\|^{\frac{1}{2}}_{\bH^{1}(0,1)}|\widetilde{R}_\eps^i(\cdot,s)\|^{\frac{3}{2}}_{\bL^{2}(0,1)}\\
& \leq C \eps^{-1-\frac{\gamma}{2}}\|\widetilde{R}_\eps^i(\cdot,s)\|^{\frac{1}{2}}_{\bH^{1}(0,1)}\|\widetilde{R}_\eps^i(\cdot,s)\|^{\frac{3}{2}}_{\bL^{2}(0,1)} \\
& \leq \frac{1}{2}  \|\widetilde{R}_\eps^i(\cdot,s)\|^{2}_{\bH^{1}(0,1)} + C\eps^{-\frac{4}{3} (1 + \frac{\gamma}{2})} \|\widetilde{R}_\eps^i(\cdot,s)\|^{2}_{\bL^{2}(0,1)},
\ena
$$
where we have used again the assumption that $\eps^\gamma \|\bA_\eps\|^2_{\bH^1(0,1)} < \infty$ in the penultimate inequality and Young's inequality and equivalence of norm on $\bH^1_0(0,1)$ in the last inequality. Hence for $\gamma \in (0, \frac{1}{2})$ and $\eps$ sufficiently small, the first term on the right side of \eqref{eq:grt-a} can be absorbed by the left hand side. 
This implies that 
\be\label{eq:grt-b}
\bea
 |\widetilde{R}_\eps^i(\cdot,s) |_{\bH^1(0,1)}^2 + \frac{a^2}{\eps^2} \|\widetilde{R}^i_\eps(\cdot,s)\|_{\bL^2(0,1)}^2 & \lesssim \int_0^1 |\bF_\eps(t, s)\cdot \widetilde{R}^i_\eps(t,s)| dt \\
 & \leq \|\bF_\eps(\cdot, s)\|_{\bL^1(0,1)}\|\widetilde{R}^i_\eps(\cdot,s)\|_{\bL^\infty(0,1)}.
\ena
\en
In addition, according to Lemma \ref{lem:ineq-1}, 
$$
\bea
 \frac{a}{\eps}\|\widetilde{R}^i_\eps(\cdot,s)\|_{\bL^\infty(0,1)}^2 & \leq \frac{2a}{\eps}|\widetilde{R}^i_\eps(\cdot,s)|_{\bH^1(0,1)}\|\widetilde{R}^i_\eps(\cdot,s)\|_{\bL^2(0,1)} \\
& \leq |\widetilde{R}^i_\eps(\cdot,s)|_{\bH^1(0,1)}^2 + \frac{a^2}{\eps^2} \|\widetilde{R}^i_\eps(\cdot,s)\|_{\bL^2(0,1)}^2 \\
& \lesssim \|\bF_\eps(\cdot, s)\|_{\bL^1(0,1)}\|\widetilde{R}^i_\eps(\cdot,s)\|_{\bL^\infty(0,1)}.
\ena 
$$
Therefore we have
\be\label{eq:grt-c}
\|\widetilde{R}^i_\eps(\cdot,s)\|_{\bL^\infty(0,1)} \lesssim \eps \|\bF_\eps(\cdot, s)\|_{\bL^1(0,1)}.
\en
This together with \eqref{eq:fepsl1} yields the estimate \eqref{eq:grtinfty}. Finally, the estimate \eqref{eq:grt2} follows from \eqref{eq:grt-b}, \eqref{eq:grt-c} and \eqref{eq:fepsl1}.
\end{proof}
As a consequence of Lemma \ref{lem:grt},
\be\label{eq:gcomp}
\bG_\eps(t,t) = \frac{\eps}{2} \bA^{-1}_\eps(t) + \eps\bR_\eps(t) + \widetilde{\bR}_\eps (t, t)
\en
where 
\be\label{eq:gcompr}
|\bR_\eps(t)|_F \leq C (e^{-\frac{2at}{\eps}} + e^{-\frac{2a(1-t)}{\eps}})
\en and $\widetilde{\bR}_\eps$ satisfies the estimates in Lemma \ref{lem:grt}.
In particular, we have
\be\label{eq:greenF}
|\bG_\eps(t,t)|_F \leq C\eps \text{ for any } t\in (0,1).
\en
 Furthermore, we obtain an asymptotic formula for the expectation term in $F^{(2)}_\eps(\bA_\eps)$. 

\begin{cor}\label{cor:expt}
Let $\gamma \in (0, \frac{1}{2})$. Let $\{\bA_\eps\} \subset \bH^1_a(0,1)$ such that $$\limsup_{\eps \gt 0} \eps^\gamma \|\bA_\eps\|^2_{\bH^1_a(0,1)}  < \infty.$$ Then for $\eps$ small enough we have
\be\label{eq:expt-1}
- \frac{\eps}{4} \E^{\overline{\nu}_\eps} \left[\int_0^1  z(t)^T \bB_\eps(t) z(t) dt\right] = -\frac{1}{4}\int_0^1 \Tr(\bA_\eps(t))dt + \cO(\eps^{1 - 2\gamma}).
\en
\begin{proof}
 Inserting \eqref{eq:gcomp} into the equation \eqref{eq:f_2expterm} and noting that $\bB_\eps = \eps^{-2}\bA_\eps^2 - \eps^{-1}\bA_\eps^\prime$, we get
\be\label{eq:expterm}
\bea
& - \frac{\eps}{4} \E^{\overline{\nu}_\eps} \left[\int_0^1  z(t)^T \bB_\eps(t) z(t) dt\right] = -\frac{1}{4}\int_0^1 \Tr(\bA_\eps(t))dt \\ 
& + \frac{\eps}{4}\int_0^1 \Tr(\bA_\eps^\prime(t)\bA^{-1}_\eps(t))dt
 - \frac{1}{2\eps }\int_0^1 \Tr\left(\left( \bA_\eps^2(t) - \eps \bA_\eps^\prime(t)\right) \big(\eps\bR_\eps(t) + \widetilde{\bR}_\eps(t, t) \big)\right) dt.
\ena\en
Now we bound the last three terms on the right hand side. First, using the trace inequality 
\be\label{eq:traceineq}
\Tr(\bC\bDD) \lesssim |\bC|_F|\bDD|_F
\en 
which holds for any matrices $\bC, \bDD$, we obtain that 
\be\bea\label{eq:r1}
\left|\frac{\eps}{4}\int_0^1 \Tr(\bA_\eps^\prime(t)\bA^{-1}_\eps(t))dt\right|
\lesssim \eps \int_0^1 |\bA^\prime_\eps(t)|_F |\bA^{-1}_\eps(t)|_F dt  \lesssim \eps^{1- \frac{\gamma}{2}}.
\ena
\en
In the second inequality we used the assumption that $\bA_\eps \geq a\cdot\bI_d $ so that $\Tr(\bA_\eps^{-1}) \leq d/a$ and hence $|\bA^{-1}_\eps(t)|_F \lesssim 1$.
Next, applying Cauchy-Schwarz inequality to the last two terms on the right of \eqref{eq:expterm} and using the assumptions on $\bA_\eps$, the inequality \eqref{eq:gcompr} and Lemma \eqref{lem:grt}, we have 
\be\bea\label{eq:r23}
& \left| \frac{1}{2\eps}\int_0^1 \Tr\left(\left( \bA_\eps^2(t) - \eps \bA_\eps^\prime(t)\right)\left( \eps\bR_\eps(t) + \widetilde{\bR}_\eps(t, t)\right)\right) dt\right| \\
& \lesssim
\eps^{-1}\|\bA_\eps\|_{\bL^\infty(0,1)}^2 \left(\eps \|\bR_\eps\|_{\bL^1(0,1)} + \int_0^1 |\widetilde{\bR}_\eps(t, t)| dt \right) \\
& \qquad \qquad + \|\bA_\eps^\prime\|_{\bL^2(0,1)} \left(\eps \|\bR_\eps\|_{\bL^2(0,1)} + \left(\int_0^1  |\widetilde{\bR}_\eps(t, t)|^2 dt\right)^{\frac{1}{2}} \right) \\
 & \lesssim \eps^{-1 - \gamma}(\eps^2 + \eps^{2-\gamma}) + \eps^{-\frac{\gamma}{2}}(\eps^{\frac{3}{2}} + \eps^{2-\gamma}) \lesssim \eps^{1 - 2\gamma},
\ena\en
where we have also used the assumption that $\gamma \in (0,\frac{1}{2})$.
This finishes the proof.
\end{proof}
\end{cor} 
We proceed to proving bounds for the logarithmic term appearing in $F^{(2)}_\eps(\bA_\eps)$.
\begin{lem}\label{lem:log}
Let $\gamma \in (0,\frac{1}{2})$.  Let $\{\bA_\eps\} \subset \bH^1_a(0,1)$ such that $$\limsup_{\eps \gt 0} \eps^\gamma \|\bA_\eps\|_{\bH^1(0,1)}^2 < \infty.$$ Then when $\eps$ is small enough
\be\label{eq:log}
C\eps\log\eps\leq \frac{\eps}{2}\log\left(\det\left(\int_0^1 \overline{\bM}_\eps(t)\overline{\bM}_\eps(t)^T dt\right)\right) \leq 0.
\en
\begin{proof}
We first prove the non-positiveness. Since $\overline{\bM}_\eps(t) = \bM_\eps(1, t)$ where the fundamental matrix $\bM_\eps$ satisfies \eqref{eq:M} with $\bA$ replaced by $\bA_\eps$. Then the $i$-th column of $\overline{\bM}_\eps$, denoted by $M_\eps^i$, satisfies 
$$
\partial_t M_\eps^i(t, s) = -\eps^{-1}\bA_\eps(t) M_\eps^{i}(t,s), \quad M_\eps^i (s, s) = e^i,
$$
where $e^i$ is the unit basis vector of $\R^d$ in the $i$-th direction. Taking the dot product of the above equation with $M^i_\eps(t,s)$ and then integrating from $s$ to $t$ implies that
$$
|M_\eps^i(t,s)|^2  = -\frac{2}{\eps} \int_s^t M_\eps^i(r,s)^T \bA_\eps(r)M_\eps^i(r,s) dr \leq -\frac{2a}{\eps}\int_s^t |M_\eps^i(r,s)|^2 dr.
$$
Consequently, $|M_\eps^i(t,s)| \leq e^{-\frac{a(t-s)}{\eps}}$ for any $0\leq s\leq t \leq 1$. Hence each entry of $\overline{\bM}_\eps(t)$ can be bounded from above by $e^{-\frac{a(1-t)}{\eps}}$. As a result, for sufficiently small $\eps$, we have
$$
\det\left(\int_0^1 \overline{\bM}_\eps(t)\overline{\bM}_\eps(t)^T dt\right) \leq C\eps < 1.
$$
The upper bound of \eqref{eq:log} thus follows. On the other hand, applying the determinant inequality \eqref{eq:det-2} to the matrix function $\overline{\bM}_\eps(t) \overline{\bM}_\eps(t)^T$ and the equality \eqref{eq:det-3} yields
\be\label{eq:logl}
\bea
 \frac{\eps}{2}\log\left(\det\left(\int_0^1 \overline{\bM}_\eps(t)\overline{\bM}_\eps(t)^T dt\right)\right) & \geq 
\frac{\eps d}{2}\log\left(\int_0^1 \det\left(\overline{\bM}_\eps(t)\right)^{\frac{2}{d}} dt\right) \\
 & = \frac{\eps d}{2} \log\left(\int_0^1 \exp\left(-\frac{2}{\eps d} \int_t^1 \Tr( \bA_\eps(s)) ds \right) dt\right).
\ena
\en
Moreover, from the assumption that $\eps^\gamma \|\bA_\eps\|_{\bH^1(0,1)}^2 < \infty$ and the fact that $\bH^1(0,1)$ is embedded into $\bL^\infty(0,1)$, we obtain that
$$
\int_t^1 \Tr(\bA_\eps(s)) ds \leq (1 - t)\|\bA_\eps\|_{\bL^\infty(0,1)} \leq C \eps^{-\frac{\gamma}{2}}(1 - t).
$$
 Combining this with \eqref{eq:logl} gives
\be\label{eq:logl-1}
\bea
 \frac{\eps}{2}\log\left(\det\left(\int_0^1 \overline{\bM}_\eps(t)\overline{\bM}_\eps(t)^T dt\right)\right) & \geq 
 \frac{\eps d}{2}\log\left(\int_0^1 \exp\left(-\frac{2C}{\eps^{1 + \frac{\gamma}{2}} d}(1 - t) \right) dt\right)\\
 & = \frac{\eps d}{2} \log\left(\frac{\eps^{1 + \frac{\gamma}{2}}d}{2C}\left(1 - e^{-\frac{2C}{\eps^{1 + \frac{\gamma}{2}} d}}\right)  \right)\\
 & \geq C \eps \log \eps
 \ena
\en
for sufficiently small $\eps$. This completes the proof.
\end{proof}
\end{lem}
Recall that the definition of $F_\eps^{(2)}$ in \eqref{eq:f1}. Then the following proposition, containing the asymptotic expression for $F^{(2)}_\eps(\bA_\eps)$, is a direct consequence of Corollary \ref{cor:expt} and Lemma \eqref{lem:log}. 

\begin{prop}\label{prop:f2}
Let $\{\bA_\eps\} \subset \bH^1_a(0,1)$ such that $\limsup_{\eps \gt 0} \eps^\gamma \|\bA_\eps\|_{\bH^1(0,1)}^2 < \infty$ with $\gamma \in (0, \frac{1}{2})$. Then it holds that
\be
F_\eps^{(2)}(\bA_\eps) = \frac{1}{4} \int_0^1 \Tr(\bA_\eps(t)) dt + \cO(\eps^{1 - \gamma}).
\en
when $\eps$ is small enough.
\end{prop}

\subsection{Asymptotics of $F^{(1)}_\eps(m_\eps, \bA_\eps)$}

In this subsection, we seek an asymptotic expression for $F^{(1)}_\eps(m_\eps, \bA_\eps)$ when it is uniformly bounded with respect to $\eps$. We start by showing that the boundedness of $F_\eps^{(1)}(m_\eps, \bA_\eps)$ implies the boundedness of $\|m_\eps\|_{\bL^\infty(0,1)}$.

\begin{lem}\label{lem:bound}
Assume that $(m_\eps, \bA_\eps) \in \HH $ and that $\limsup_{\eps \gt 0} F_\eps^{(1)}(m_\eps, \bA_\eps) < \infty$. Then we have $\limsup_{\eps \gt 0} \|m_\eps\|_{\bL^\infty(0,1)} < \infty$.

\begin{proof}
Recalling that $\Psi_\eps(x) = \frac{1}{2} |\nabla V(x)|^2 - \eps \Delta V(x)$ and that $\nu_\eps = N(m_\eps, \bSig_\eps)$, we can rewrite $F^{(1)}_\eps(m_\eps, \bA_\eps)$ as 
$$\bea
F^{(1)}_\eps(m_\eps, \bA_\eps) &=
\frac{\eps}{4}\int_0^1 |m^\prime (t)|^2 dt + \frac{1}{2\eps} \E^{\overline{\nu}_\eps} \left[\int_0^1 \Psi_\eps(z(t) + m_\eps(t)) dt\right]\\
& = F^{(3)}_\eps(m_\eps, \bA_\eps) + F^{(4)}_\eps(m_\eps, \bA_\eps)
\ena
$$
where 
\ben
 F_\eps^{(3)}(m_\eps, \bA_\eps) := \frac{\eps}{4}\int_0^1 |m_\eps^\prime (t)|^2 dt + \frac{1}{8\eps} \int_0^1 \E^{\overline{\nu}_\eps} \left[ |\nabla V(z(t) + m_\eps(t))|^2 \right] dt
\enn
and 
\ben\bea
 F_\eps^{(4)}(m_\eps, \bA_\eps) & := \frac{1}{8\eps} \int_0^1 \E^{\overline{\nu}_\eps} \left[ |\nabla V(z(t) + m_\eps(t))|^2 - 4\eps \Delta V(z(t) + m_\eps(t)) \right] dt\\
 & = \frac{1}{8\eps} \int_0^1 \E^{\nu_\eps} \left[ |\nabla V(x(t))|^2 - 4\eps \Delta V(x(t)) \right] dt.
\ena\enn
First, from \eqref{eq:psilb} of Remark \ref{rem:cond} we can obtain immediately that 
$$
\liminf_{\eps \gt 0} F_\eps^{(4)}(m_\eps, \bA_\eps)> -\infty.
$$
This together with the assumption that $\limsup_{\eps \gt 0} F_\eps^{(1)}(m_\eps, \bA_\eps) < \infty$ implies
$$\limsup_{\eps \gt 0} F_\eps^{(3)}(m_\eps, \bA_\eps) < \infty.$$
We now show that this implies the uniformly boundedness of $\|m_\eps\|_{\bL^\infty(0,1)}$.

We prove a lower bound for $F_\eps^{(3)}(m_\eps, \bA_\eps)$. 
Given any $R>0$, define $T^R_\eps := \{t\in (0,1): |m_\eps(t)| > R\}$ which is an open set on $(0,1)$. By restricting the second integral and expectation over a smaller set, we have
\be\label{eq:bdd1}
F_\eps^{(3)}(m_\eps, \bA_\eps) \geq \frac{\eps}{4}\int_0^1 |m_\eps^\prime (t)|^2 dt + \frac{1}{8\eps}\int_{T^R_\eps} \E^{\overline{\nu}_\eps} \left[|\nabla V(m_\eps(t) + z(t))|^2 \mathbf{1}_{\{|z(t)|\leq \eps^{1/4}\}} \right] dt.
\en
Consider $(t, \omega)$ such that $|m_\eps(t)| > R$ and $|z(t,\omega)| \leq \eps^\frac{1}{4}$. If $\eps > 0$ is small enough to satisfy $\eps^\frac{1}{4} < R / 2$, then 
$$|m_\eps(t) + z(t,\omega)| \geq |m_\eps(t)| - |z(t,\omega)| \geq |m_\eps(t)|/2. $$
Combining this with the monotonicity condition (A-6) yields that 
\be\label{eq:bdd2}
\bea
& \frac{1}{8\eps}\int_{T^R_\eps} \E^{\overline{\nu}_\eps} \left[|\nabla V(m_\eps(t) + z(t))|^2 \mathbf{1}_{\{|z(t)|\leq \eps^{\frac{1}{4}}\}} \right] dt \\
& \geq \frac{1}{8\eps}\int_{T^R_\eps}  |\nabla V(m_\eps(t)/2)|^2 \overline{\nu}_\eps \left(\{|z(t)|\leq \eps^{\frac{1}{4}}\} \right) dt 
 \geq \frac{1}{16\eps}\int_{T^R_\eps} |\nabla V(m_\eps(t)/2)|^2 dt 
\ena
\en
when $\eps > 0$ is small enough. We have used the fact that $\overline{\nu}_\eps  \left(\{|z(t)|\leq \eps^{\frac{1}{4}}\} \right) \geq 1/2$ for any $t\in (0,1)$ and small $\eps$. This is because $z(t)$ is a centred Gaussian random variable with covariance $2\bG_\eps(t,t)$ (see \eqref{eq:greenexp}). In addition, we know from \eqref{eq:greenF} that $|\bG_\eps(t,t)|_F \leq C\eps$ for any $t\in (0,1)$ and hence 
$\overline{\nu}_\eps  \left(\{|z(t)|\leq \eps^{\frac{1}{4}}\} \right) \gt 1$ when $\eps \gt 0$. Let $\widetilde{m}_\eps = m_\eps/2$.  From \eqref{eq:bdd1}, \eqref{eq:bdd2} and the uniform boundedness of $F_\eps^{(3)}(m_\eps, \bA_\eps)$ we obtain
$$
\limsup_{\eps \gt 0} \eps \int_{T^R_\eps} |\widetilde{m}_\eps^\prime(t)|^2 dt + \frac{1}{16\eps}\int_{T^R_\eps} |\nabla V(\widetilde{m}_\eps(t))|^2 dt < \infty.
$$
Then application of the elementary inequality $2ab \leq a^2 + b^2$ yields
$$
\limsup_{\eps \gt 0} \int_{T^R_\eps} |\widetilde{m}_\eps^\prime(t)| |\nabla V(\widetilde{m}_\eps(t))| dt < \infty.
$$
Choosing a sufficiently large $R$ and by the coercivity condition (A-3), we have
\be\label{eq:bdd3}
\limsup_{\eps \gt 0} \int_{T^R_\eps} |\widetilde{m}_\eps^\prime(t)| < \infty.
\en
Now we conclude the uniform boundedness of $\|m_\eps\|_{L^\infty(0,1)}$ by applying the same argument used for proving Theorem 1.2 in \cite{L14}. Specifically, since $m_\eps$ is continuous on $(0,1)$, $T^R_\eps$ is open on $(0,1)$ and we can write $T^R_\eps = \cup_{i=1}^\infty (a^i_\eps, b^i_\eps)$. Suppose that $T^R_\eps$ is empty, then $|m_\eps(t)| \leq R$ for all $t\in (0,1)$. Otherwise, consider $\widetilde{m}_\eps(t)$ with $t\in (a^i_\eps, b^i_\eps)$. Obviously at least one of the end points of the subinterval, say $a^i_\eps$ is not an endpoint of $(0,1)$. Then we should have $|m_\eps(a^i_\eps)| = R$ and hence $|\widetilde{m}_\eps(a^i_\eps)| = 2R$. Thus we get from the fundamental theorem of calculus that
$$
\bea
\limsup_{\eps \gt 0} \sup_{t\in (a^i_\eps, b^i_\eps)} |\widetilde{m}_\eps(t)| & \leq  \limsup_{\eps \gt 0} \left( |\widetilde{m}_\eps(a^i_\eps)| + \sup_{t\in (a^i_\eps, b^i_\eps)} \Big|\int_{a^i_\eps}^t \widetilde{m}_\eps(s) ds \Big|\right)\\
& \leq 2R + \limsup_{\eps \gt 0} \int_{T^R_\eps} |\widetilde{m}_\eps^\prime(t)| < \infty
\ena
$$
where the last inequality follows from \eqref{eq:bdd3}. Therefore $\lim\sup_{\eps \gt 0}\|m_\eps\|_{L^\infty(0,1)} < \infty$.
\end{proof}
\end{lem}

Next, the expectation term of $F_\eps^{(1)}(m_\eps, \bA_\eps)$ can be simplified under the condition that $\|m_\eps\|_{\bL^\infty(0,1)}$ is uniformly bounded. 
\begin{lem}\label{lem:F1}
Let $(m_\eps, \bA_\eps) \in \HH$. Assume that $\limsup_{\eps \gt 0} \eps^\gamma \|\bA_\eps\|^2_{\bH^1(0,1)} < \infty$ with $\gamma\in (0, \frac{1}{2})$ and that $\limsup_{\eps \gt 0} \|m_\eps\|_{\bL^\infty(0,1)} < \infty$. Then for $\eps>0$ small enough we have
$$
\bea
& F_\eps^{(1)}(m_\eps, \bA_\eps)  = E_\eps(m_\eps)  - \frac{1}{2}\int_0^1 \Delta V(m_\eps(t))dt\\
& + \frac{1}{4} \int_0^1 \left(D^2 V(m_\eps(t))^2 + D^3 V(m_\eps(t))\cdot \nabla V(m_\eps(t))  \right): \bA_\eps^{-1}(t) dt + \cO(\eps^{\frac{1}{2}}).
\ena
$$

\begin{proof} 
Remember that 
$$F_\eps^{(1)}(m_\eps, \bA_\eps) = \frac{\eps}{4}\int_0^1 |m^\prime (t)|^2 dt + \frac{1}{2\eps} \E^{\overline{\nu}_\eps} \left[\int_0^1 \Psi_\eps(m_\eps(t) + z_\eps(t)) dt\right].$$
To evaluate the expectation term of $F_\eps^{(1)}(m_\eps, \bA_\eps)$, we use the following multi-variable Taylor's formula for $\Psi_\eps$: 
$$
\Psi_\eps(x(t)) = \Psi_\eps(m_\eps(t)) + \nabla \Psi_\eps(m_\eps(t)) \cdot z_\eps(t) + \frac{1}{2} z_\eps(t)^T D^2\Psi_\eps(m_\eps(t))z_\eps(t)+ r_\eps(t),
$$
where the reminder term $r_\eps$ is given in integral form by
$$r_\eps(t) = \sum_{|\alpha| = 3} \frac{z_\eps^\alpha(t)}{\alpha!} \int_{0}^1 \partial^\alpha \Psi_\eps(m_\eps(t) + \xi z_\eps(t))(1 - \xi)^3 d\xi. $$
Here $\alpha = (\alpha_1,\alpha_2, \cdots, \alpha_d)$ is a multi-index and we use the notational  convention 
 $x^\alpha = x_1^{\alpha_1}x_2^{\alpha_2}\cdots x_d^{\alpha_d}$ and 
 $\partial^\alpha f = \partial_1^{\alpha_1} \partial_2^{\alpha_2} \cdots \partial_d^{\alpha_d} f$.  
Then using again the fact that $z_\eps(t) \sim N(0,2\bG_\eps(t,t))$, we obtain that
\be\label{eq:taylor}
\bea
& \frac{1}{2\eps} \E^{\overline{\nu}_\eps} \left[\int_0^1 \Psi_\eps(m_\eps(t) + z_\eps(t)) dt\right] \\
& = \frac{1}{2\eps} \int_0^1 \Psi_\eps(m_\eps(t)) dt
 + \frac{1}{2\eps}\int_0^1 D^2 \Psi_\eps(m_\eps(t)): \bG_\eps(t, t) dt + \frac{1}{2\eps} \int_0^1 \E^{\overline{\nu}_\eps} \left[r_\eps(t)\right]dt.
\ena
\en
Recalling that $\Psi_\eps(x) = \frac{1}{2}|\nabla V(x)|^2 - \eps \Delta V(x)$, we have
$$
D^2 \Psi_\eps(x) = (D^2 V(x))^2 + D^3 V(x)\cdot \nabla V(x) - \eps D^2 (\Delta V(x)).
$$
 From this equation, the expression \eqref{eq:gcomp} for $\bG_\eps(t,t)$ and the uniform boundedness of $\|m_\eps\|_{\bL^\infty(0,1)}$, the second term on the right side of \eqref{eq:taylor} becomes
 \be
 \bea\label{eq:talor2nd}
&\frac{1}{2\eps}\int_0^1 D^2 \Psi_\eps(m_\eps(t)): \bG_\eps(t, t) dt \\
&= \frac{1}{4}\int_0^1 \left(\big(D^2 V(m_\eps(t))\big)^2 + D^3 V(m_\eps(t))\cdot \nabla V(m_\eps(t))\right) : \bA_\eps^{-1}(t) dt + \cO(\eps).
 \ena
\en
Next we claim that the integral of the last term on the right hand side of \eqref{eq:taylor} is of order $\mathcal{O}(\eps^{\frac{3}{2}})$. Indeed, from the assumption \eqref{eq:G} and the fact that $z_\eps(t) = N(0, 2\bG_\eps(t,t))$ with $\bG_\eps(t,t)$ satisfying the estimate \eqref{eq:greenF}, we have 
\be\label{eq:residue3}
\bea
& \E^{\nu_\eps}[r_\eps(t)] \leq \sum_{|\alpha| = 3} \frac{1}{\alpha!} \max_{\xi \in [0,1]} \left\{\E^{\nu_\eps} \left[|z_\eps(t)|^3 \partial_\alpha \Psi_\eps(m_\eps(t) + \xi z_\eps(t))\right]\right\} \\
& \leq \frac{C_1}{\sqrt{(4\pi)^d \det(2\bG_\eps(t,t))}} \max_{\xi \in [0,1]} \left\{\int_{\R^d} |x|^3 e^{C_2|m_\eps(t) + \xi x|^\alpha}\cdot e^{-\frac{1}{4} x^T\bG_\eps(t,t)^{-1}x} dx \right\}\\
& \leq \frac{C_1}{\sqrt{(4\pi)^d \det(\bG_\eps(t,t))}} e^{2C_2 \|m_\eps\|_{\bL^\infty(0,1)}^2}  \int_{\R^d} e^{2C_2 |x|^\alpha} |x|^3  e^{-\frac{1}{4} x^T \bG_\eps(t,t)^{-1} x} dx \\
& \leq \frac{C_2}{\sqrt{(4\pi)^d}} |\bG_\eps(t,t)|_F^{\frac{3}{2}} \cdot e^{2C_2 \|m_\eps\|^2_{\bL^\infty(0,1)}} \cdot \int_{\R^d}  |x|^3 e^{-\frac{|x|^2}{4}}dx \leq C \eps^{\frac{3}{2}}.
\ena
\en
when $\eps$ is small enough. Notice that in last two inequalities of above we used the fact that $\alpha \in [0,2)$ and that $|\bG_\eps(t,t)|_F \leq C\eps$ so that $e^{2C_2|x|^\alpha}$ can be absorbed by $e^{\frac{1}{4} x^T \bG_\eps(t,t)^{-1} x}$ for large $x$.
Then the desired result follows from \eqref{eq:taylor}, \eqref{eq:talor2nd} and \eqref{eq:residue3}. 
\end{proof}
\end{lem}

\subsection{Proof of Main Results}\label{subsec:proofs}

\begin{proof45}
The proposition follows directly from the definition of $F_\eps$, Proposition \ref{prop:f2}, Lemma \ref{lem:F1} and the following equalities 
\be\label{eq:treq}
\bea
&-2 \Delta V + (D^2 V)^2: \bA^{-1} + \Tr(\bA) \\
& = 
-2 \Tr(D^2 V) + \Tr\left((D^2 V)^2 \bA^{-1}\right) + \Tr\left(\bA^2 \bA^{-1}\right)\\
& = -\Tr\left((\bA D^2 V \bA^{-1}\right) - \Tr\left(\bA^{-1}D^2 V \bA\right)   
 + \Tr\left((D^2 V)^2 \bA^{-1}\right) + \Tr\left(\bA^2 \bA^{-1}\right)\\
& = (D^2 V - \bA)^2 : \bA^{-1},
\ena\en
which are valid for any $V\in C^2(\R^d)$ and any positive definite matrix $\bA$.
\end{proof45}
The following lemma shows that $\eps \log(Z_{\mu,\eps})$ is bounded from above.
\begin{lem}\label{lem:constZ}
There exists $C > 0$ depending only on the potential $V$ such that the following holds: 
\be\label{eq:constZ}
\limsup_{\eps \gt 0} \eps \log \left(Z_{\mu,\eps}\right) \leq C.
\en
\end{lem}
\begin{proof}
Recall that
$$Z_{\mu,\eps} = \E^{\mu_0}\left[ \exp\left(-\frac{1}{2\eps^2} \int_0^1 |\nabla V(x(t))|^2 - \eps \Delta V(x(t)) dt \right)\right].$$
From \eqref{eq:psilb} of Remark \ref{rem:cond}, 
$$
Z_{\mu,\eps} \leq \exp\left(\frac{C}{\eps}\right)
$$
with some $C > 0$.
This proves \eqref{eq:constZ}.
\end{proof}

\begin{prooffoursix}
Assume that $\limsup_n F_{\eps_n} (m_n, \bA_n) < \infty$. Since the Kullback-Leibler divergence $D_{\text{KL}}(\nu_{\eps_n} || \mu_{\eps_n})$ is always non-negative, it follows from \eqref{eq:kld} and Lemma \ref{lem:constZ} that 
\be
\liminf_{n\gt \infty} \eps_n \widetilde{D}_{\text{KL}}(\nu_{\eps_n} || \mu_{\eps_n}) \geq -C
\en
for some $C > 0$.
 This together with the assumption that $\limsup_n F_{\eps_n} (m_n, \bA_n) < \infty$ implies that $\lim\sup_n \eps_n^\gamma \|\bA_{n}\|_{\bH^1(0,1)}^2 < \infty$. Then from Proposition \ref{prop:f2} and noting that $\bA(\cdot) - a\cdot \bI_d \geq 0$, we obtain 
$$
\limsup_n F_{\eps_n}^{(2)}(\bA_n) \geq \limsup_n \frac{1}{4} \int_0^1 \Tr(\bA_n(t)) dt \geq \frac{d a}{4}.
$$
Hence we have $\limsup_n F_{\eps_n}^{(1)} (m_n, \bA_n) < \infty$. Then Lemma \ref{lem:bound} implies that 
$$\limsup_n \|m_n\|_{\bL^\infty(0,1)} < \infty.$$ Hence as a consequence of Lemma \ref{lem:F1},
 \be\label{eq:fepsn}
\bea
F_{\eps_n}(m_{n}, \bA_n) & = E_{\eps_n}(m_{n})  + \frac{1}{4}\int_0^1\left(D^2 V(m_n(t)) - \bA_n(t)\right)^2: \bA_n^{-1}(t) dt \\
& + \int_0^1 (D^3 V(m_n(t))\cdot \nabla V(m_n(t)) ): \bA_n^{-1}(t) dt + \eps_n^\gamma \|\bA_n\|_{\bH^1(0,1)}^2 + \cO(\eps_n^{\frac{1}{2}}).
\ena
 \en
The second term on the right side of above is nonnegative. In addition, owing to the trace inequality \eqref{eq:traceineq} and the fact that $\bA_n \geq a\cdot \bI_d$,
\be\bea\label{eq:nabla3VA1}
& \limsup_n \left|\frac{1}{4} \int_0^1  D^3 V(m_{n}(t))\cdot \nabla V(m_n(t)): \bA_{n}^{-1}(t) dt \right|\\
& \leq \limsup_n  \int_0^1 \left|\big(D^3 V(m_n(t))\cdot \nabla V(m_n(t))\big)^2\right|_F \left|\bA_n^{-1}(t)\right|_F dt  < \infty.
\ena\en
This implies from \eqref{eq:fepsn} that $\limsup_n E_{\eps_n}(m_n) < \infty$. By the compactness result of Lemma \ref{lem:compact-1}, there exists $m\in \BV(0,1; \EE)$ and a subsequence $m_{n_k}$ such that $m_{n_k} \gt m$ in $\bL^1(0,1)$. Moreover, we know from the above reasoning that 
\be\label{eq:nabla2V}
\limsup_{n} \frac{1}{4}\int_0^1\left(D^2 V(m_n(t)) - \bA_n(t)\right)^2: \bA_n^{-1}(t) dt
 < \infty
\en
from which we can conclude that $\sup_{n} \|\bA_n\|_{\bL^1(0,1)} < \infty$. Indeed, 
\be\label{eq:nabla2V-2}
\bea
& \int_0^1\left(D^2 V(m_n(t)) - \bA_n(t)\right)^2: \bA_n^{-1}(t) dt \\ 
& = \int_0^1 \Tr\left(\big(D^2 V(m_n(t))\big)^2 \bA_n^{-1}(t)\right)dt - 2\int_0^1 \Tr\left(D^2  V(m_n(t))\right) dt
 + \int_0^1 \Tr(\bA_n(t)) dt.
\ena\en
The first term on the right of above is non-negative. The second term is clearly bounded since $\|m_n\|_{\bL^\infty(0,1)}$ is uniformly bounded. Hence $\sup_{n} \|\bA_n\|_{\bL^1(0,1)} < \infty$ follows from \eqref{eq:nabla2V}, \eqref{eq:nabla2V-2} and the inequality $|\bA|_F \leq \Tr(\bA)$ which holds for any positive definite matrix $\bA$. 
\end{prooffoursix}

The proof of $\Gamma$-limit of $F_\eps$ is presented in what follows. 

\begin{proof4main}
We start by proving the liminf inequality, i.e. 
$$F(m, \bA) \leq \liminf_{\eps \gt 0}  F_\eps(m_\eps, \bA_\eps)$$
 for any sequence $\{(m_\eps, \bA_\eps)\}$ such that $(m_\eps, \bA_\eps) \gt (m, \bA)$ in $\mathcal{X}$, or equivalently $m_\eps \gt m$ and $\bA_\eps \wgt \bA$ in $\bL^1(0,1)$. We may assume that $\liminf_{\eps \gt 0} F_\eps(m_\eps, \bA_\eps) < \infty$ since otherwise there is noting to prove. Then by the same argument used in the proof of Proposition \ref{prop:compact-2}, one can get $\limsup_{\eps\gt 0} \eps^\gamma \|\bA_\eps\|^2_{\bH^1(0,1)} < \infty$. Let $\{\eps_k\}$ be a sequence such that $\eps_k \gt 0$ as $k\gt \infty$ and $\lim_{k\gt \infty} F_{\eps_k}(m_{\eps_k}, \bA_{\eps_k}) = \liminf_{\eps \gt 0} F_\eps(m_\eps, \bA_\eps) < \infty$. Since $\bA_{\eps_k} \geq a\cdot \bI_d$ a.e., it follows from $\bA_{\eps_k} \wgt \bA$ and Mazur's lemma (Lemma \ref{lem:mazur}) that the limit $\bA\geq a\cdot \bI_d$ a.e. According to Proposition \ref{prop:f2} and $\bA_{\eps_k} \wgt \bA$ in $\bL^1(0,1)$, it holds that 
 $$\lim_{k\gt \infty} F_{\eps_k}^{(2)}(\bA_{\eps_k}) = \frac{1}{4}\int_0^1 \Tr(\bA(t))dt \geq \frac{da}{4}.$$ 
 Then it follows that $\lim_{k\gt \infty} F_{\eps_k}^{(1)}(m_{\eps_k}, \bA_{\eps_k}) < \infty$. From Lemma \ref{lem:bound} we obtain that $\|m_{\eps_k}\|_{\bL^\infty(0,1)}$ is uniformly bounded. Hence as a consequence of Lemma \ref{lem:F1},
 $$
\bea
& F_{\eps_k}^{(1)}(m_{\eps_k}, \bA_{\eps_k}) = E_\eps(m_{\eps_k})  - \frac{1}{2}\int_0^1 \Delta V(m_{\eps_k}(t))dt \\
& + \frac{1}{4} \int_0^1 \left(D^2 V(m_{\eps_k}(t))^2 + D^3 V(m_{\eps_k}(t))\cdot \nabla V(m_{\eps_k}(t))  \right): \bA_{\eps_k}^{-1}(t) dt + O(\eps_k^{\frac{1}{2}}).
\ena
 $$
 In addition, it follows from the uniform boundedness of $\|m_{\eps_k}\|_{\bL^\infty(0,1)}$ and $\bA_{\eps_k} \wgt \bA$ in $\bL^1(0,1)$ that 
 $$\bea
& \limsup_{k\gt \infty} \Big\{\Big|- \frac{1}{2}\int_0^1 \Delta V(m_{\eps_k}(t))dt \\
& + \frac{1}{4} \int_0^1 \left(D^2 V(m_{\eps_k}(t))^2 + D^3 V(m_{\eps_k}(t))\cdot \nabla V(m_{\eps_k}(t))  \right): \bA_{\eps_k}^{-1}(t) dt\Big|\Big\} < \infty. 
\ena
 $$
This in turn implies that $\limsup_{k\gt \infty} E_{\eps_k}(m_{\eps_k}) < \infty$. By the compactness result in Lemma \ref{lem:compact-1}, we have $m\in \BV((0,1); \EE)$. Furthermore, by passing to a subsequence, we may assume further that $m_{\eps_k} \gt m$ a.e. on $[0,1]$. 
 Since $m$ takes value in $\EE$ a.e. on $[0,1]$, we use the definition of $D^3 V$ and the dominated convergence theorem to conclude that 
\be\label{eq:nablaV31}
\int_0^1 \big(D^3 V(m_{\eps_k}(t))\cdot \nabla V(m_{\eps_k}(t))\big): \bA_{\eps_k}^{-1}(t) dt \gt 0.
\en 
In fact, similar to \eqref{eq:nabla3VA1}, we have
$$
\int_0^1 \Big| \big(D^3 V(m_{\eps_k}(t))\cdot \nabla V(m_{\eps_k}(t))\big): \bA_{\eps_k}^{-1}(t) \Big| dt < \infty.
$$
In addition, since $\nabla V(m_{\eps_k}(t)) \gt 0$ a.e. on $[0,1]$ and $\bA_{\eps_k} \geq a\cdot \bI_d$, we have
\be\label{eq:nabla3V-3}
\bea
& \Big| \big(D^3 V(m_{\eps_k}(t))\cdot \nabla V(m_{\eps_k}(t))\big): \bA_{\eps_k}^{-1}(t) \Big|\\
 & \leq
 \Big|D^3 V(m_{\eps_k}(t))\cdot \nabla V(m_{\eps_k}(t))\Big|_F \Big| \bA_{\eps_k}^{-1}(t) \Big|_F
  \gt 0.
\ena\en
a.e. on $[0,1]$. This proves \eqref{eq:nablaV31}. Now we claim that
\be\label{eq:lowsemicont}
\bea
&\int_0^1 \left(D^2 V(m(t)) - \bA(t)\right)^2: \bA^{-1}(t) dt \\
& \leq \liminf_{k \gt \infty} \int_0^1 \big(D^2 V(m_{\eps_k}(t)) - \bA_{\eps_k}(t)\big)^2 : \bA_{\eps_k}^{-1}(t)dt.
\ena
\en
when $m_{\eps_k} \gt m$ in $\bL^1(0,1)$ with $m(\cdot)\in \EE$ a.e. and $\bA_{\eps_k} \wgt \bA$ in $\bL^1(0,1)$.
In fact, the weak convergence $\bA_{\eps_k} \wgt \bA$ in $\bL^1(0,1)$ directly implies that
$$
\int_0^1 \bA_{\eps_k}^2(t) :\bA_{\eps_k}^{-1}(t)  - \bA^2(t) :\bA^{-1}(t)dt = \int_0^1 \Tr(\bA_{\eps_k}(t) - \bA(t)) dt \gt 0.
$$
In addition, thanks to the uniform boundedness of $\|m_{\eps_k}\|_{\bL^\infty(0,1)}$ and the strong convergence $m_{\eps_k}\gt m$ in $\bL^1(0,1)$, we obtain from the dominated convergence that
$$
\bea
& \int_0^1 \left(D^2 V(m_{\eps_k}(t)\bA_{\eps_k}(t) +  \bA_{\eps_k}(t)D^2 V(m_{\eps_k}(t)) \right) : \bA_{\eps_k}^{-1}(t)dt  \\
& - \int_0^1  \left(D^2 V(m(t) \bA(t) +  \bA(t) D^2 V(m(t)) \right) : \bA^{-1}(t) dt \\
& = 2\int_0^1 \Tr(D^2 V(m_{\eps_k}(t) - D^2 V(m(t)) ) dt \gt 0.
\ena
$$
Therefore to prove \eqref{eq:lowsemicont}, it suffices to show 
\be\label{eq:liminfnabla2V}
\int_0^1 D^2(m(t))^2 : \bA^{-1}(t) dt \leq \liminf_{k\gt \infty} \int_0^1 D^2 V(m_{\eps_k}(t))^2: \bA_{\eps_k}^{-1}(t) dt =: \theta.
\en
To that end, let $\bB(\cdot):= D^2 V(m(\cdot))^2$. Noting that $m(\cdot) \in \EE$ a.e, we know from (A-2) of Assumptions \eqref{assump} that $\bB(\cdot)$ is positive definite a.e. Define the functional 
\be\label{eq:functionalM}
\Ms (\bA) = \int_0^1 \bB(t):\bA^{-1}(t)dt
\en
over the set $\bL^1_a(0,1)$. Then \eqref{eq:liminfnabla2V} becomes 
\be\label{eq:liminfnabla2V-2}
\Ms (\bA) \leq \liminf_{k \gt \infty} \Ms (\bA_k).
\en Note that $\bL^1_a(0,1)$ is a convex subset of the space $\bL^1(0,1)$. We first claim that the functional $\mathscr{M}$ is convex on $\bL^1_a(0,1)$. In fact, for any $\bA_1, \bA_2 \in \bL_a^1(0,1), \alpha\in (0,1)$,
$$
\bea
\Ms \left(\alpha \bA_1 + (1 - \alpha) \bA_2\right)
& = \int_0^1 \Tr\Big(\bB(t) \big(\alpha \bA_1(t) + (1 - \alpha) \bA_2(t)\big)^{-1}\Big) dt \\
& = \int_0^1 \Tr\Big(\alpha \bA_1(t) \bB^{-1}(t) + (1 - \alpha) \bA_2(t) \bB^{-1}(t)\Big)^{-1} dt\\
& \leq \alpha \int_0^1 \Tr\big( \bA_1(t)\bB^{-1}(t)  \big)^{-1} dt +  (1 - \alpha)\int_0^1 \Tr\big(  \bA_2(t)\bB^{-1}(t)  \big)^{-1} dt \\
& = \alpha \Ms (\bA_1) + (1 - \alpha)\Ms (\bA_2),
\ena
$$
where we used the trace inequality $\Tr((\bC + \bDD)^{-1}) \leq \Tr(\bC^{-1}) + \Tr(\bDD^{-1})$ for positive definite matrices $\bC, \bDD$. Now we prove \eqref{eq:liminfnabla2V-2} by employing the convexity of $\Ms$. First by passing a subsequence (without relabeling), we may assume that $\Ms (\bA_k)$ converges to $\theta$.
According to Mazur's Lemma \ref{lem:mazur}, there exists a convex combination of $\{\bA_{k}\}$, defined by
$$
\overline{\bA}_j = \sum_{k = j}^{N(j)} \alpha_{j,k} \bA_k, \quad \alpha_{j,k} \in [0,1], \quad \sum_{k = j}^{N(j)} \alpha_{j,k} = 1,
$$
such that $\overline{\bA}_j \gt A$ strongly in $\bL^1(0,1)$. Note that we applied Mazur's Lemma \ref{lem:mazur} to the sequence $\{\bA_{k}\}_{k\geq j}$ at step $j$. Since $\Ms$ is convex, we obtain
$$
\Ms (\overline{\bA}_j)  = \Ms \left(\sum_{k = j}^{N(j)} \alpha_{j,k} \bA_k\right)
\leq \sum_{k = j}^{N(j)} \alpha_{j,k} \Ms (\bA_k).
$$
Letting $j \gt \infty$, since $k \geq j$ in the sum and $\Ms (\bA_k) \gt \theta$, we have
\be\label{eq:MAjAk}
\lim\inf_{j\gt \infty} \Ms (\overline{\bA}_j) \leq \theta = \lim\inf_{k\gt \infty} \Ms (\bA_k).
\en
In addition, it holds that 
\be\label{eq:MAAj}
\Ms (\overline{\bA}_j) \gt \Ms (\bA).
\en
Indeed, since $m\in \BV(0,1;\EE)$ and $\overline{\bA}_j \gt \bA$ in $\bL^1(0,1)$, 
$$
\bea
\big|\Ms \left(\bA - \overline{\bA}_j\right) \big|&= \Big|\int_0^1 \Tr\Big(D^2 V(m(t))\big(\bA^{-1}(t) - \overline{\bA}_j^{-1}(t) \big)\Big) dt\Big|\\
& \leq \int_0^1 \Big|\Tr\Big(D^2 V(m(t)) \bA^{-1}(t)\big(\overline{\bA}_j(t) - \bA(t) \big)\overline{\bA}_j^{-1}(t) \Big)\Big|dt\\
& \lesssim \int_0^1 \Big|D^2 V(m(t)) \bA^{-1}(t)\big(\overline{\bA}_j(t) - \bA(t) \big)\overline{\bA}_j^{-1}(t)  \Big|_F dt\\
& \lesssim \int_0^1 |D^2 V(m(t))|_F |\bA^{-1}(t)|_F |\overline{\bA}_j(t) - \bA(t)|_F |\overline{\bA}_j^{-1}(t)|_F dt\\
&\lesssim \|\overline{\bA}_j(t) - \bA(t)\|_{\bL^1(0,1)}\gt 0.
\ena
$$
Therefore \eqref{eq:liminfnabla2V-2} follows from \eqref{eq:MAjAk} and \eqref{eq:MAAj} and thereby proves \eqref{eq:lowsemicont}.

Taking account of the fact that $E_\eps$ $\Gamma$-converges to $E$, we obtain from Proposition \ref{prop:sfeps}, \eqref{eq:nablaV31} and \eqref{eq:lowsemicont} that
\ben
\bea
& \liminf_{\eps \gt 0} F_\eps (m_\eps, \bA_\eps)  = \lim_{k\gt \infty}  F_{\eps_k} (m_{\eps_k}, \bA_{\eps_k}) \\
& \geq E(m)  + \frac{1}{4}\int_0^1 \left(D^2 V(m(t)) - \bA(t)\right)^2: \bA^{-1}(t) dt = F(m, \bA).
\ena
\enn

Next we prove the limsup inequality, i.e. for a subsequence $\eps_k \gt 0$, we want to find a pair of recovering sequence $(m_{k}, \bA_{k})$ converging to $(m, A)$ such that $$\limsup_{k\gt \infty} F_{\eps_k}(m_{k}, \bA_{k}) \leq F(m, \bA).$$ It suffices to deal with the case where $F(m, \bA) < \infty$ and hence $m(t)\in \EE$. Otherwise the limsup inequality is automatically satisfied. First thanks to the $\Gamma$-convergence of $E_\eps$ to $E$, one automatically obtains a recovering sequence $m_{k}\in \bH^1_\pm(0,1)$ such that $m_k \gt m$ in $\bL^1(0,1)$, $\limsup_k \|m_k\|_{\bL^\infty(0,1)} < \infty$ and $\limsup_{k\gt \infty} E_\eps(m_{k}) \leq E(m)$. We construct a recovering sequence $\bA_{k} \in \bH^1_a(0,1)$ explicitly by using convolution approximation. Specifically fixing any $\alpha <\gamma/3 $, we define 
\be\label{eq:gconv}
\bA_{k} := \widetilde{\Kc}_{\eps_k^{\alpha}} \bA
\en
where $\widetilde{\Kc}_{\eps}$ is the convolution operator defined in \eqref{eq:tildeK}. It is proved in Lemma \ref{lem:convapp-2} that $\bA_k \in \bH^1_a(0,1)$ and $\bA_k \gt \bA$ in $\bL^1(0,1)$. Moreover, by replacing $\eps$ with $\eps^{\alpha}$ in the bound proved in Lemma \eqref{lem:convapp-2}, we have
\be\label{eq:Akh1}
\eps_k^\gamma \|\bA_k\|_{\bH^1(0,1)}^2 \lesssim \eps_k^{\gamma -3\alpha} \gt 0.
\en 
 With the above choices for $m_k$ and $\bA_k$, we get from Proposition \ref{prop:sfeps} that
 $$
 \bea
& \limsup_{k\gt \infty} F_{\eps_k}(m_{k}, \bA_{k}) \\
& = \limsup_{k\gt \infty}\Big\{ E_{\eps_k}(m_{k}, \bA_{k}) \\
& \qquad + \frac{1}{4}\int_0^1 \left(D^2 V(m_{k}(t)) - \bA_{k}(t)\right)^2: \bA_{k}^{-1}(t) dt\\
& \qquad + \int_0^1  \big(D^3 V(m_{k}(t))\cdot  \nabla V(m_{k}(t))\big): \bA_{k}^{-1}(t) dt + \eps_k^\gamma \|\bA_{k}\|_{\bH^1(0,1)}^2\Big\}\\
& \leq E(m) + \frac{1}{4}\int_0^1 \left(D^2 V(m_k(t)) - \bA_k(t)\right)^2 : \bA_k^{-1}(t) dt \\
&= F(m, \bA).
 \ena
 $$
 To pass to the inequality we have used the dominated convergence theorem and \eqref{eq:nablaV31} for the third  term on the left hand side as well as  \eqref{eq:Akh1} for the fourth term on the left hand side. The proof is now complete.
\end{proof4main}

\begin{proof4cor}
Let $(m_\eps, \bA_\eps)\in \HH$ be a minimizer of $F_\eps$. We first argue that $\limsup_\eps F_\eps(m_\eps, \bA_\eps) < \infty$. In fact, for any fixed $m\in \BV(0,1; \EE)$ and $\bA\in \bL^1_a(0,1)$, we know from the proof of the limsup inequality of Theorem \ref{thm:main} that there exists a recovering sequence $(\widetilde{m}_\eps, \widetilde{\bA}_\eps)$ such that $\limsup_\eps F_\eps(\widetilde{m}_\eps, \widetilde{\bA}_\eps) < \infty$. Since $(m_\eps, \bA_\eps)$ minimizes $F_\eps$, we have 
$$
\limsup_{\eps\gt 0} F_\eps(m_\eps, \bA_\eps) \leq \limsup_{\eps\gt 0} F_\eps(\widetilde{m}_\eps, \widetilde{\bA}_\eps) < \infty.
$$
Then by Proposition \eqref{prop:compact-2}, there exists a subsequence $\eps_k$ and the corresponding $\{(m_k, \bA_k)\}\subset \HH$ such that $m_k \gt m$ in $\bL^1(0,1)$ and $\bA_k \wgt \bA$ in $\bL^1(0,1)$ with some $m\in \BV(0,1;\EE)$ and $\bA \in \bL_a^1(0,1)$. We now show that $(m, \bA)$ minimizes $F$. 
In fact, given any $\widetilde{m}\in \BV(0,1;\EE)$ and $\widetilde{\bA}\in \bL^1_a(0,1)$, thanks to the $\Gamma$-convergence of $F_\eps$ to $F$, one can find a recovering sequence $(\widetilde{m}_k, \widetilde{\bA}_k)\in \HH$ such that 
$$\lim\sup_k F_{\eps_k}(\widetilde{m}_k, \widetilde{\bA}_k) \leq F(\widetilde{m}, \widetilde{\bA}).$$
Since $(m_k, \bA_k)$ minimizes $F_{\eps_k}$, we have $F_{\eps_k}(m_k, \bA_k) \leq F_{\eps_k}(\widetilde{m}_k, \widetilde{\bA}_k)$. Then using the liminf inequality part of the $\Gamma$-convergence of $F_\eps$ to $F$, we obtain
$$
F(m, \bA) \leq \liminf_{k\gt \infty} F_{\eps_k} (m_k, \bA_k) \leq \limsup_{k\gt \infty} F_{\eps_k}(\widetilde{m}_k, \widetilde{\bA}_k) \leq F(\widetilde{m}, \widetilde{\bA}).
$$
Since $\widetilde{m}, \widetilde{\bA}$ is arbitrary, $(m, \bA)$ is a minimizer of $F$.  

\end{proof4cor}




\appendix

\section{Estimates for The Constant Coefficient Green's Functions}
Assume that function $|\lambda(t)| \geq a$ almost everywhere in $[0,1]$ with some $a > 0$. For any $s\in (0,1)$, let $\overline{G}^\lambda_\eps$ be the solution to 
the equation
\be\label{eq:gr2}
\begin{aligned}
& \left(-\partial_t^2 + \eps^{-2} \lambda^2(s)\right) \overline{G}^\lambda_\eps(t,s) = \delta(t - s),\quad t\in (0,1),\\
& \overline{G}^\lambda_\eps(0, s) = \overline{G}^\lambda_\eps(1, s) = 0.
\end{aligned}
\en
where $\delta$ is the Dirac function. The solution $\overline{G}^{\lambda}_{\eps}$ is given explicitly as follows
\ben
\overline{G}^\lambda_\eps(t, s) = \frac{\eps}{|\lambda(s)| \sinh(|\lambda(s)|/\eps)} \times
 \begin{cases}
 \sinh(|\lambda(s)|s/\eps) \sinh(|\lambda(s)|(1-t)/\eps) & s \leq t;\\
 \sinh(|\lambda(s)|t/\eps) \sinh(|\lambda(s)|(1-s)/\eps) & s \geq t.
\end{cases}
\enn
Notice that $\overline{G}^\lambda_\eps(t,s)$ is not a standard Green's function as it is not symmetric with respect to permutation of its arguments. According to the definition of $\sinh$, a few elementary calculations yield the following estimates.

\begin{lem}
Let $|\lambda(t)| \geq a$ a.e. on $(0,1)$ for a fixed $a > 0$. Then for sufficiently small $\eps > 0$, the solution $\overline{G}^\lambda_\eps$ to the equation \eqref{eq:gr2} satisfies the following.
\begin{enumerate}
\item[(i)] There exists $C = C(a) > 0$ such that
\be\label{eq:gr3}
0 \leq \overline{G}^\lambda_\eps(t, s) \leq C\eps e^{-\frac{a}{\eps} |s-t|}
\en
 for any $t,s \in [0,1]$. 
\item[(ii)]  There exists $C = C(a) > 0$ such that  
 $$\overline{G}^\lambda_\eps (t,t)= \frac{\eps}{2}\left(\frac{1}{|\lambda(t)|} + R(t)\right), \quad t \in [0,1]
 $$
  with 
 \be\label{eq:gr-lam}
|R(t)| \leq C \Big(e^{-\frac{2a t}{\eps}} + e^{-\frac{2a(1 - t)}{\eps}}\Big).
 \en
 \end{enumerate}
\end{lem}

 Considering the Green's tensor $\overline{\bG}_\eps(t, s)$ that solves the matrix equation
\be\label{eq:grt-2}
\bea
(-\bD + \eps^{-2}\bA(s)) \overline{\bG}_\eps(\cdot, s) = \delta(\cdot - s) \cdot \bI_d,\\
 \overline{\bG}_\eps(0, s) =  \overline{\bG}_\eps(1, s) = 0,
\ena
\en
with $\bA\in \bH^1(0,1)$ and $|\bA| \geq a$ a.e. on $(0,1)$, we have the following similar estimates. 
\begin{lem}\label{lem:a2}
Let $\bA\in \bH^1(0,1)$ and $|\bA| \geq a$ a.e. on $(0,1)$. For sufficiently small $\eps > 0$, the solution $\overline{\bG}_\eps$ to the equation \eqref{eq:grt-2} satisfies the following.
\begin{enumerate}
\item[(i)] there exists $C = C(a) > 0$ such that 
\be\label{eq:grt-3}
|\overline{\bG}_\eps(t, s)|\leq C \eps e^{-\frac{a}{\eps} |s - t|}
\en
for any $t, s\in [0,1]$.
\item[(ii)] there exists $C = C(a) > 0$ such that 
\be\label{eq:grt-4}
\overline{\bG}_\eps(t,t) = \frac{\eps}{2}\left( |\bA^{-1}(t)| + \bR(t)\right)
\en
with 
\be
|\bR(t)| \leq C \Big(e^{-\frac{2a t}{\eps}} + e^{-\frac{2a(1 - t)}{\eps}}\Big).
\en
\end{enumerate}
\begin{proof}
Since $\bA(s)$ is symmetric for any $s\in (0,1)$ (by the definition of $\bH^1(0,1)$), there exists an orthogonal matrix $\bP(s)$ such that $\bA(s) = \bP^{-1}(s) \mathbf{\Lambda}(s) \bP(s)$ where $\mathbf{\Lambda}(s) = \text{diag}(\lambda_1(s), \cdots, \lambda_d(s))$. Moreover, by assumption we have $|\lambda_i(s)| \geq a$ a.e. on $(0,1)$ for any $i = 1,\cdots, d$. Therefore, the problem \eqref{eq:grt-2} can be diagonalized so that one obtains 
\begin{equation}\label{e:AppendixA-diagonalisation}
\overline{\bG}_\eps(t, s) = \bP^{-1}(s) \cdot \text{diag}(\overline{G}^{\lambda_1}_\eps(t,s), \cdots, \overline{G}^{\lambda_d}_\eps(t, s)) \cdot \bP(s),
\end{equation}
where $\overline{G}^{\lambda_i}_\eps(\cdot, s)$ solves \eqref{eq:gr2} with $\lambda$ replaced by $\lambda_i$. Then \eqref{eq:grt-3} follows directly from \eqref{e:AppendixA-diagonalisation} and equation \eqref{eq:gr3}, and \eqref{eq:grt-4} can be deduced from  \eqref{e:AppendixA-diagonalisation} and \eqref{eq:gr-lam}.
\end{proof}
\end{lem}

\section{Fundamental Matrix of Linear Systems}\label{appendix-B}
Given $f: \R \gt \R^d$ and $\bA : \R \gt \R^{d\times d}$, consider the following linear differential equation
\be\label{eq:ode-1}
d x_\eps (t) = -\eps^{-1} \bA(t) x_\eps(t)dt + f(t)dt, \quad x_\eps (t_0) = 0.
\en
The solution to \eqref{eq:ode-1} can be found via the variation of constants method provided its fundamental matrix is determined.
\begin{defn}[Fundamental matrix]\label{def:ch5fundm}
The fundamental matrix $\bM_\eps(t,t_0)$ is the solution matrix that solves the problem 
\be\label{eq:ode-2}
\frac{d}{d t} \bM_\eps (t, t_0) = -\eps^{-1} \bA(t) \bM_\eps(t, t_0), \quad \bM_\eps (t_0, t_0) = \bI_d.
\en
\end{defn}
Suppose that $\bA$ and $f$ are both continuous, then the solution to the ODE \eqref{eq:ode-1} can be written in the form
$$
x_\eps(t) = \int_{t_0}^t \bM_\eps (t, s) f(s) ds.
$$
We comment that the above formula is still valid when $f(s)ds$ is replaced by $dW(s)$, in
which case the integral is understood as It\^o's stochastic integration. In the case that $d = 1$ or if $\bA$ does not depend on $t$, we have
$\bM_\eps(t, s) = \exp\big(-\eps^{-1}\int_s^t \bA(r)dr\big)$. In general, there is
no closed form expression for the fundamental matrix $\bM_\eps$ and hence the solution to \eqref{eq:ode-1}
has no explicit formula. Nevertheless, $\bM_\eps$ has some nice properties which are useful to
study the asymptotic behavior of the solution to \eqref{eq:ode-1} when $\eps \gt 0$.
\begin{lem}
Let $\bM_\eps$ be the fundamental matrix defined by \eqref{eq:ode-2}. Then we have
\begin{enumerate}
\item[(i)] For all $t, t_0, t_1\in \R$, $\bM_\eps(t, t_0) = \bM_\eps(t, t_1) \bM_\eps(t_1, t_0)$. 
\item[(ii)] For all $t, t_0\in \R$, $\bM_\eps(t, t_0)$ is non-singular and $\bM_\eps^{-1}(t, t_0) = \bM_\eps(t_0, t)$.
\item[(iii)] For all $t, t_0\in \R$,
\be\label{eq:det-3}
\det(\bM_\eps(t, t_0)) = \exp\Big(-\eps^{-1}\int_{t_0}^t \Tr(\bA(s))ds\Big).
\en
\end{enumerate}
\end{lem}
\begin{proof}
The proof can be found in \cite[Chapter 6]{BN67}.
\end{proof}

We finish this appendix with two useful inequalities about the determinants of symmetric positive definite matrices.

\begin{lem}
If $\bA, \bB$ are real symmetric positive definite matrices of size $d$, then
\be\label{eq:det-1}
(\det(\bA + \bB))^{\frac{1}{d}} \geq (\det(\bA))^{\frac{1}{d}} +  (\det(\bB))^{\frac{1}{d}}.
\en
\end{lem}
A proof of this lemma can be found in \cite[Page 115]{MM92}. It shows that the function $\bA \mapsto \det(\bA)^{\frac{1}{d}}$ is concave. As a consequence, we have the following corollary.
\begin{cor}
Suppose that $\bA\in C([0,1]; \R^{d\times d})$ is a matrix-valued function and that $\bA(t)$ is symmetric positive definite for any $t \in[0, 1]$. Then we have
\be\label{eq:det-2}
\left[\det\left(\int_0^1 \bA(t)dt \right)\right]^{\frac{1}{d}} \geq \int_0^1 \det(\bA(t))^{\frac{1}{d}} dt.
\en
\end{cor}
\begin{proof}
Define $t_i = i/N$ with $i = 0, 1, \cdots, N$. Since $\bA$ is continuous on $[0,1]$, it holds that
$$
\bea
 \left[\det\left(\int_0^1 \bA(t)dt \right)\right]^{\frac{1}{d}}  &= \lim_{N\gt \infty} \left[\det\left(\frac{1}{N}\sum_{i=1}^N \bA(t_i) \right)\right]^{\frac{1}{d}}\\
&  = \lim_{N\gt \infty} \frac{1}{N} \left[\det\left(\sum_{i=1}^N \bA(t_i) \right)\right]^{\frac{1}{d}} \\
 & \geq \lim_{N\gt \infty} \frac{1}{N} \sum_{i=1}^N \det(\bA(t_i))^{\frac{1}{d}} = \int_0^1 \det(\bA(t))^{\frac{1}{d}} dt.
\ena
$$
We have used \eqref{eq:det-1} in the inequality.
\end{proof}

\section{Useful Inequalities and Lemmas}
\begin{lem}[Interpolation inequality]\label{lem:int}
Let $u\in H^1_0(0,1)$. Then
\be\label{eq:interp} 
\|u\|_{H^s} \lesssim \|u\|_{H^1}^{s} \|u\|_{L^2}^{1-s}
\en
for any $s\in [0, 1]$. 
\end{lem}
\begin{proof}
See \cite[Corollary 6.11]{M09} for the proof.
\end{proof}

\begin{lem}\label{lem:ineq-1}
Let $u\in H_0^1(0,1)$. Then $\|u\|_{L^\infty}^2 \leq 2\|u\|_{L^2} |u|_{H^1} $. 
\begin{proof}
It suffices to prove the inequality when $u\in C_0^\infty(0,1)$. For any $t\in (0,1)$, it follows from the fundamental theorem of calculus and Cauchy-Schwarz inequality that
\be
u^2(t) = 2\int_0^t u(s) u^\prime(s) ds \leq 2\|u\|_{L^2}|u|_{H^1}.
\en
The lemma then follows by taking the supremum over $t$.
\end{proof}
\end{lem}


\begin{lem}\label{lem:app1}
Let $\alpha, \beta \in (0,1)$ be such that $\beta > \max(\alpha, \alpha/2 + 1/4)$. Then any matrix-valued function $B\in \bH^{-\alpha}(0,1)$ can be viewed as a bounded multiplication operator from $\bH_0^{\beta}$ to $\bH^{-\beta}$. Furthermore we have
\be\label{eq:multi1}
\|B\|_{\mathcal{L}(\bH_0^{\beta}, \bH^{-\beta})} \lesssim \|B\|_{\bH^{-\alpha}(0,1)}. 
\en
\end{lem}
\begin{proof} It suffices to consider the proof in the scalar case.
Let $B\in H^{-\alpha}$ and $\varphi\in H_0^\beta$. Assume that $\beta > \alpha$, then one can define the multiplication $B \varphi$ as a distribution in the sense that for any $\psi \in C_0^\infty(0,1)$
$$
\langle B\varphi, \psi\rangle = \langle B, \varphi \psi\rangle
$$
Moreover, if $\beta - \alpha/2 > 1/4$, we have
$$ 
|\langle B\varphi, \psi\rangle| = |\langle B, \varphi \psi\rangle| \leq \|B\|_{H^{-\alpha}} \|\varphi \psi\|_{H^\alpha} \lesssim \|B\|_{H^{-\alpha}} \|\varphi \|_{H^\beta} \|\psi \|_{H^\beta}
$$
where the last estimate follows from the following Lemma. Therefore the desired estimate \eqref{eq:multi1} holds. 
\end{proof}

\begin{lem}
Let $\alpha, \beta$ and $\gamma$ be positive exponents such that $\min(\alpha, \beta) > \gamma$ and $\alpha + \beta > \gamma + 1/2$. Then, if $\varphi \in H^\alpha$ and $\psi\in H^\beta$, the product $\varphi\psi$ belongs to $H^\gamma$ and $\|\varphi\psi\|_{H^\gamma} \lesssim \|\varphi\|_{H^\alpha} \|\psi\|_{H^\beta}$. 
\end{lem}
\begin{proof} The proof can be found in \cite[Theorem 6.18]{M09}.
\end{proof}

\begin{lem}\label{lem:schodinger}
Let $A\in H^1(0,1)$ and let $f\in L^2(0,1)$. Set $B = A^2 - A^\prime$. Then there exits a unique solution $u\in H_0^1(0,1)$ solving the problem 
\be\bea\label{eq:probschodinger}
& (-\partial_t^2  + B) u = f \text{ on } (0,1),\\
& u(0) = u(1) = 0.
\ena\en
Moreover, it holds that 
\be\label{eq:schodingerest}
\|u\|_{H^{2}(0,1)} \leq C \|f\|_{L^2(0,1)},
\en
and $(-\partial_t^2  + B)^{-1}$ is a trace-class operator on $L^2(0,1)$.
\end{lem}
\begin{proof}
Let $G_0(s, t)$ be the Dirichlet Green's function of $-\partial_t^2$ on $(0,1)$. In fact, $G_0(s, t) = s(1 - t) \wedge t(1 - s)$ for all $s, t\in [0,1]$. From Green's first identity, it is easy to observe that
 a solution $u\in H_0^1(0,1)$ solving \eqref{eq:probschodinger} is a solution $u\in L^2(0,1)$ that solves the Lippmann-Schwinger integral equation
 \be\label{eq:LSieq}
u(t) + \int_0^1 G_0(t, s) B(s) u(s) ds = \int_0^1 G_0(t, s) f(s)ds
 \en
 and vice versa. Now we apply the Fredholm alternative theorem to prove the existence and uniqueness of solution to \eqref{eq:LSieq}. First the operator 
 $$
 (\mathcal{T} u) (t) := (-\partial_t^2)^{-1}(B u)(t)= \int_0^1 G_0(t, s) B(s) u(s)ds
 $$
 is compact from $L^2(0,1)$ to itself. There are several ways to prove this, but the simplest argument is perhaps the observation that $ \mathcal{T}$ is bounded from $L^2(0,1)$ to $W^{1,\infty}(0,1)$. Indeed, since $G_0(t, \cdot) \in W^{1,\infty}(0,1)$ for any $t\in [0,1]$, we can apply Cauchy-Schwarz inequality twice to get
 \be\bea\label{eq:boundTu}
\|\mathcal{T} u\|_{W^{1,\infty}(0,1)} & \leq \sup_{t\in [0,1]}\|G_0(t,\cdot)\|_{W^{1,\infty}(0,1)} \|B u\|_{L^1(0,1)} \\
& \leq \sup_{t\in [0,1]}\|G_0(t,\cdot)\|_{W^{1,\infty}(0,1)} \|B\|_{L^2(0,1)} \|u\|_{L^2(0,1)}.
 \ena\en
 Then the compactness of $\mathcal{T}$ follows from the compact embedding $W^{1,\infty}(0,1) \hookrightarrow L^2(0,1)$. We are left to show the uniqueness of \eqref{eq:LSieq} or equivalently \eqref{eq:probschodinger}.  To see this, setting $f = 0$, we multiply the equation \eqref{eq:probschodinger} by $u$, integrate, use Green's first identity and get 
$$
 \bea
 0 & = \int_0^1 u^\prime(t)^2 dt + \int_0^1 (A^2(t) - A^\prime(t)) u^2(t) dt \\
 & = \int_0^1 u^\prime(t)^2 dt + \int_0^1 A^2(t) u^2(t) dt 
 + 2 \int_0^1 A(t)u(t)u^\prime(t) dt\\
 & = \int_0^1 (u^\prime(t) + A(t)u(t))^2 dt.
 \ena
$$
 Therefore we should have $u^\prime(t) = -A(t) u(t)$. The only solution to this equation with the Dirichlet boundary conditions is zero. Hence by the Fredholm alternative theorem, the integral equation \eqref{eq:LSieq} has a unique solution in $L^2(0,1)$. Then the estimate \eqref{eq:schodingerest} follows from \eqref{eq:boundTu}, \eqref{eq:LSieq} and estimate that $\|\mathcal{R}f\|_{H^{2}(0,1)} \leq C\|f\|_{L^2(0,1)}$ where
 $$
 (\mathcal{R}f)(t) = \int_0^1 G_0(t,s) f(s) ds.
$$
Finally observe that 
$$
(-\partial_t^2 + B)^{-1} = \big(I + (-\partial_t^2)^{-1} B\big)^{-1} (-\partial_t^2)^{-1} = \big(I + \mathcal{T} \big)^{-1} (-\partial_t^2)^{-1}.
$$
Then it follows from the fact that $(-\partial_t^2)^{-1}$ is a trace-class operator on $L^2(0,1)$ and the boundedness of $(I + \mathcal{T})^{-1}$ that $(-\partial_t^2 + B)^{-1}$ is trace-class.

\end{proof}

\begin{rem}\label{rem:schodinger}
Lemma \ref{lem:schodinger} can be easily extended to the matrix-valued case. More precisely, assume that $\bA\in \bH^1(0,1)$ and $\bA(t)$ is a symmetric matrix for any $t\in [0,1]$. Let $\bB = \bA^2 - \bA^\prime$. Then the inverse of the matrix-valued Schr\"odinger operator $(-\bD + \bB)^{-1}$ is bounded from $\bL^2(0,1)$ to $\bH^{2}(0,1)$ and is a trace-class operator on $\bL^2(0,1)$. $\qed$
\end{rem}

The next lemma discusses some properties of approximation by convolution. 
\begin{lem}\label{lem:convapp-1}
Let $K\in C_0^\infty(\R)$ such that $K \geq 0$ and $\int_\R K = 1$. Denote by $K_\eps(\cdot) = \eps^{-1} K(x/\eps)$. Suppose that $f\in L^1(\R)$ and define $\mathcal{K}_\eps f = K_\eps \ast f$. Then $\mathcal{K}_\eps f \in L^1(\R)\cap C^\infty(\R)$. Moreover, we have 
\be\label{eq:appidentity}
\mathcal{K}_\eps f \gt f \quad \text{ in }  L^1(\R)
\en
and 
\be\label{eq:apph1norm}
\|\mathcal{K}_\eps f\|_{H^1(\R)} \leq C\eps^{-\frac{3}{2}}\|f\|_{L^1(\R)}.
\en
\begin{proof}
The property \eqref{eq:appidentity}, often termed as the approximation of identity in $\L^1(\R)$, has been proved in many books, e.g. \cite{BN71}. We now show that $\mathcal{K}_\eps f\in H^1(\R)$ and that \eqref{eq:apph1norm} is valid. This can be seen from the observation that
$$
\bea
\|\mathcal{K}_\eps f\|_{H^1(\R)}^2 & = \|\mathcal{K}_\eps f \|_{L^2(\R)}^2 + \|\mathcal{K}_\eps^\prime f \|_{L^2(\R)}^2  \\
& = \|\eps^{-1}K(\cdot/\eps) \ast f \|^2_{L^2(\R)}  + 
\eps^{-2} \|\eps^{-1}K^\prime(\cdot/\eps) \ast f \|^2_{L^2(\R)} \\
& \leq \eps^{-1}\|K\|_{L^2(\R)}^2\|f\|_{L^1(\R)}^2 + \eps^{-3}\|K^\prime\|_{L^2(\R)}^2\|f\|_{L^1(\R)}^2
\leq C\eps^{-3} \|f\|_{L^1(\R)}^2.
\ena
$$
Note that we have used Young's inequalities in the penultimate inequality. 
\end{proof}
\end{lem}
 We continue to adapt Lemma \eqref{lem:convapp-1} to matrix functions defined on a bounded domain. For this purpose, we define two useful operators. Given a function $f\in L^1(0,1)$, we define its extension 
 $$
\Ec f (x) := \begin{cases}
 f(x) & \text{ if } x\in (0,1)\\
 0 & \text{ otherwise }.
 \end{cases}
 $$
Conversely, for a function $g\in L^1(\R)$, we define the restriction
$\Rc f := f|_{(0,1)}$. Likewise, we can define the convolution, extension or restriction of a matrix function through entry-wise operations. The following lemma concerns the convolution approximation of matrix-valued functions.
\begin{lem}\label{lem:convapp-2}
Let $\bA\in \bL^1_a(0,1)$. Define
\be\label{eq:tildeK}
\widetilde{\Kc}_\eps \bA := \Rc \left(\mathcal{K}_\eps \left(\Ec (\bA - a\cdot\bI_d)\right) + a\cdot \bI_d\right).
\en
Then $\widetilde{\Kc}_\eps \bA \in \bH^1_a(0,1)$. Moreover, $\widetilde{\Kc}_\eps \bA \gt \bA$ in $\bL^1(0,1)$ and $\|\widetilde{\Kc}_\eps \bA\|_{\bH^1(0,1)} \leq C \eps^{-\frac{3}{2}}$ with the constant $C$ depending on $A$ and $a$. 
\end{lem}
\begin{proof}
First it follows from Lemma \ref{lem:convapp-1} that $\widetilde{\Kc}_\eps \bA \in \bH^1(0,1)$. To show $\widetilde{\Kc}_\eps \bA \in \bH^1_a(0,1)$, it suffices to show $\mathcal{K}_\eps \left(\Ec (\bA - a\cdot\bI_d)\right)$ is positive semi-definite. Indeed, for any fixed $x\in \R^d$,
$$\bea
x^T \K_\eps \left(\Ec \big(\bA - a\cdot\bI_d\big)\right) x & = \K_\eps\left(\Ec \big(x^T(\bA - a\cdot\bI_d)x \big)\right) \\
& = K_\eps(\cdot) \ast \Ec \big(x^T(\bA(\cdot) - a\cdot\bI_d)x \big) \geq 0
\ena
$$
where we have used the assumption that $\bA(\cdot) - a\cdot\bI_d$ is positive semi-definite a.e. on $(0,1)$. Next from Lemma \ref{lem:convapp-1} and the fact that $\Ec \big(\bA - a\cdot\bI_d\big)\in \bL^1(\R)$,  we have
$$
\bea
\|\widetilde{\Kc}_\eps \bA - \bA\|_{\bL^1(0,1)} 
& = \|\Rc\left(\K_\eps \left(\Ec \big(\bA - a\cdot\bI_d\big)\right) - \Ec \big(\bA - a\cdot\bI_d\big)\right)\|_{\bL^1(0,1)}\\
& \leq
 \|\K_\eps \left(\Ec \big(\bA - a\cdot\bI_d\big)\right) - \Ec \big(\bA - a\cdot\bI_d\big)\|_{\bL^1(\R)}
 \gt 0.
\ena
$$
By similar arguments one can show that $\|\widetilde{\Kc}_\eps \bA\|_{\bH^1(0,1)} \leq C\eps^{-\frac{3}{2}}$.
\end{proof}

The next lemma characterizes explicitly for the minimizer of the second component (as a functional of $\bA$) of the functional $F$ defined in \eqref{e:main-functional-1}. Recall the notation $|\bA|$ of a matrix $\bA$ defined in Section \ref{subsection:notation}.

\begin{lem}\label{lem:secondofF}
Let $\bB$ be a fixed symmetric matrix. Let $\bA$ be a minimizer of the functional 
 $$
 \mathcal{G} (\bA) := (\bB - \bA)^2: \bA^{-1}
 $$ 
over all positive matrices. Then it holds that $\bA = |\bB|$. With this choice of $\bA$, $ \mathcal{G} (\bA) = 2(\Tr(|\bB|)  - \Tr(\bB))$.
\begin{proof}
 The lemma follows from the fact that the functional $\mathcal{G}$ can be rewritten as 
 $$
 \bea
\mathcal{G} (\bA) & = \Tr(\bB^2 \bA^{-1}) + \Tr(\bA) - 2\Tr(\bB) \\
& = \Tr(\bB^2 \bA^{-1}) + \Tr(\bA) - 2 \Tr (|\bB|) + 2\Tr(|\bB|) - 2\Tr(\bB)\\
& = \Tr\left((|\bB| - \bA)^2 \bA^{-1}\right) + 2\Tr(|\bB|) - 2\Tr(\bB)\\
 & = (|\bB| - \bA)^2 : \bA^{-1} + 2\Tr(|\bB|) - 2\Tr(\bB).
\ena 
 $$
\end{proof}
\end{lem}

Finally the following Mazur's Lemma is useful to obtain a strong convergent subsequence from a weakle convergent sequence. The proof can be found in \cite[Corollary 3.8]{B10}.

\begin{lem}\label{lem:mazur}
\textbf{(Mazur's lemma)} Let $X$ be a Banach space and let $\{u_n\}_{n\in \N}$ be a sequence in $X$ that converges weakly to $u\in X$. Then there exists a sequence $\{\overline{u}_j\}_{j\in \N}$ defined by the convex combination of $\{u_n\}_{n\in \N}$, namely
\be
\overline{u}_j = \sum_{n = 1}^{N(j)} \alpha_{j,n} u_n, \quad \alpha_{j, n} \in [0, 1], \quad \sum_{n=1}^{N(j)} \alpha_{j, n} = 1,
\en
such that $\overline{u}_j$ converges to $u$ strongly in $X$.
\end{lem}

\section{Proof of Proposition \ref{prop:gammalimit-E}}\label{app-d}

We first show the liminf inequality. Suppose that $\{m_\eps\}\subset \bH^1_\pm(0,1)$ and that $ m_\eps \gt m \text{ in } \bL^1(0,1)$, we want to prove that $E(m) \leq \liminf_{\eps\gt 0} E_\eps(m_\eps)$. We may assume that $\liminf_{\eps\gt 0} E_\eps(m_\eps) < \infty$ since otherwise there is nothing to prove. Let $\{\eps_n\}$ and $\{m_{n}\}\subset \bH^1_\pm(0,1)$ be subsequences such that $\eps_n \gt 0$ and that
$$\lim_{n\gt \infty} E_{\eps_n}(m_n) = \liminf_{\eps\gt 0} E_\eps(m_\eps) < \infty.$$
By Lemma \ref{lem:compact-1}, $m \in \BV(0,1; \EE)$ and one can extract a further subsequence (without relabeling) such that $m_n(t) \gt m(t)$ a.e. $t\in (0,1)$. It is sufficient to deal with the case where $m$ only has a single jump at $\tau \in (0,1)$, i.e. 
\be\label{eq:m}
m(t) = 
\begin{cases}
m(\tau^-), \text{ if } t \in(0, \tau),\\
m(\tau^+), \text{ if } t\in [\tau,1).
\end{cases}
\en
Let $0 < t_1 < \tau < t_2 < 1$ and $m_n(t_1) \gt m(t_1) = m(\tau^-), m_n(t_2) \gt m(t_2) = m(\tau^+)$. Define $\widetilde{m}_n = m_n(\eps_n^{-1}(t - (t_1 + t_2)/2))$.  Then it follows from the equality \eqref{eq:eqphi} that
$$
\bea
& \lim_{n\gt \infty} \frac{1}{4}\int_{t_1}^{t_2} \eps_n \big|m_n^\prime(t)\big|^2 + \frac{1}{\eps_n}\big|\nabla V(m_n(t))\big|^2 dt\\
 &  \quad =  \lim_{n\gt \infty} \frac{1}{4}\int_{\frac{t_1 - t_2}{2\eps_n}}^{\frac{t_2 - t_1}{2\eps_n}}
|\widetilde{m}_n^\prime(t)|^2 + |\nabla V(\widetilde{m}_n(t))|^2 dt 
\\ 
 &\quad \geq   \inf_{T, m} \Big\{ \frac{1}{4}\int_{-T}^{T} |m^\prime(t)|^2 + |\nabla V(m(t))|^2 dt : T > 0, m\in \bH^1(-T, T) \text{ and }\\
 &  \quad \quad \quad m(-T) = x_-, m(T) = x_+ \Big\} = \Phi(m(\tau^-), m(\tau^+)).
\ena
$$
Similarly, taking into account that $m_n$ satisfies the end point conditions, one can obtain
$$
\lim_{n\gt \infty} \frac{1}{4}\int_{0}^{t_1} \eps_n \big|m_n^\prime(t)\big|^2 + \frac{1}{\eps_n}\big|\nabla V(m_n(t))\big|^2 dt \geq
\Phi(x_-, m(0^+))
$$
and 
$$
\lim_{n\gt \infty} \frac{1}{4}\int_{t_2}^{1} \eps_n \big|m_n^\prime(t)\big|^2 + \frac{1}{\eps_n}\big|\nabla V(m_n(t))\big|^2 dt \geq
\Phi(m(1^-), x_+).
$$
Therefore the liminf inequality $E(m) \leq \liminf_{\eps \gt 0} E_\eps(m_\eps)$ follows.

Now we prove the limsup inequality, and again it suffices to consider $m$ defined by \eqref{eq:m}. According to the equation \eqref{eq:eqphi}, for any small $\eta > 0$, there exists $T > 0$ and $m_i\in \bH^1(-T, T), i=1,2,3$ such that 
$$
\bea
&  m_1(-T) = x_-, m_1(T) = m(0^+) \text{ and } \mathcal{J}_T(m_1) \leq \Phi(x_-, m(0^+)) + \eta/3,\\
&  m_2(-T) = m(\tau^-),  m_2(T) = m(\tau^+) \text{ and } \mathcal{J}_T(m_2) \leq \Phi(m(\tau^-), m(\tau^+)) + \eta/3,\\
& m_3(-T) = m(1^-),  m_3(T) = x_+   \text{ and } \mathcal{J}_T(m_3) \leq \Phi(m(1^-), x_+) + \eta/3.
\ena
$$
Then for $\eps > 0$ small enough, we define the recovery sequence 
$$
m_\eps(t) = \begin{cases}
m_1\left(-T + \eps^{-1} t\right) & \text{ if } t \in (0, 2\eps T),\\ 
m(0^+) & \text{ if } t \in (2\eps T, \tau -  \eps T),\\ 
m_2\left(\eps^{-1}(t - \tau)\right) & \text{ if } t \in (\tau -  \eps T, \tau + \eps T),\\
m(1^-) & \text{ if } t \in ( \tau + \eps T, 1 - 2\eps T),\\
m_3\left(\eps^{-1}(t - 1) +  T\right) & \text{ if } t \in (1 - 2\eps T, 1).
\end{cases}
$$
It is clear that $m_\eps \in \bH^1_\pm(0,1)$ and $m_\eps \gt m$ in $\bL^1(0,1)$ as $\eps  \gt 0$. Furthermore, we have 
$$
\bea
 &\limsup_{\eps \gt 0}  E_\eps(m_\eps)\\
  & =\limsup_{\eps \gt 0}  \left\{\frac{1}{4}\int_0^1 \eps |m_\eps^\prime(t)|^2 + \frac{1}{\eps} |\nabla V(m_\eps(t))|^2 dt\right\}\\
& = \limsup_{\eps \gt 0} \Big\{\frac{1}{4} \int_0^{2\eps T} \eps |m_\eps^\prime(t)|^2 + \frac{1}{\eps} |\nabla V(m_\eps(t))|^2 dt \\
& + \frac{1}{4} \int_{\tau - \eps T}^{\tau + \eps T} \eps |m_\eps^\prime(t)|^2 + \frac{1}{\eps} |\nabla V(m_\eps(t))|^2 dt +\frac{1}{4} \int_{1 - 2\eps T}^1 \eps |m_\eps^\prime(t)|^2 + \frac{1}{\eps} |\nabla V(m_\eps(t))|^2 dt\Big\}\\
 & \leq \Phi(x_-, m(0^+)) + \Phi(m(\tau^-), m(\tau^+)) + \Phi(m(1^-), x_+) + \eta.
\ena
$$ 
Since $\eta$ is arbitrary, the limsup inequality follows.  
\bibliographystyle{siamplain}
\bibliography{transpathsiam}
\end{document}